\documentclass[11pt,letterpaper]{amsart}
\usepackage{amssymb,mathrsfs,amsmath,xspace}
\usepackage{color,umoline}
\usepackage[dvipsnames]{xcolor}
\usepackage{graphicx}
\usepackage{enumitem}
\usepackage[font=small,labelfont=bf]{caption}

\usepackage{tikz, caption, subcaption}

\usepackage{relsize}
\usepackage[mathscr]{eucal}
\usepackage{verbatim}
\usepackage[draft]{todonotes}
\usepackage{hyperref}

\usepackage[dvipsnames]{xcolor}
\usetikzlibrary{patterns}
\DeclareMathOperator{\rank}{rank}
\tikzstyle{EDR}=[draw=lightgray,line width=0pt,preaction={clip, postaction={pattern=north east lines, pattern color=gray}}]
\tikzstyle{EDR1}=[draw=lightgray,line width=0pt,preaction={clip, postaction={pattern=north west lines, pattern color=gray}}]

\def\wt#1{\widetilde{#1}}

\newcommand{\dist}{\operatorname{dist}}
\newcommand{\MKN}{\mathrm{MKN}}
\newcommand{\KN}{\mathrm{KN}}

\definecolor{mygray}{gray}{0.95}

\definecolor{mypink1}{rgb}{1.2,1.1,0.9}

\definecolor{mypink2}{rgb}{1.0,0.95 ,0.9}

\definecolor{mypink3}{rgb}{1.0,0.6,0.7}

\numberwithin{equation}{section}

\newtheorem{theorem}{Theorem}[section]
\newtheorem{lemma}{Lemma}[section]
\newtheorem{corollary}{Corollary}[section]
\newtheorem{conjecture}{Conjecture}[section]
\newtheorem{proposition}{Proposition}[section]
\newtheorem{definition}{Definition}[section]
\newtheorem{example}{Example}[section]
\theoremstyle{remark}
\newtheorem{remark}{Remark}[section]

\DeclareRobustCommand{\H}{\ifmmode\text{\textup{\bfseries H1)}}\else\textup{\bfseries H1)}\fi\xspace}
\DeclareRobustCommand{\HH}{\ifmmode\text{\textup{\bfseries H2)}}\else\textup{\bfseries H2)}\fi\xspace}
\DeclareRobustCommand{\HHH}{\ifmmode\text{\textup{\bfseries H2+)}}\else\textup{\bfseries H2+)}\fi\xspace}

\newcommand{\R}{\mathbb R}
\newcommand{\ZR}{\mathbb R}
\newcommand{\TT}{\mathbb T}
\newcommand{\ZT}{\mathbb T}
\newcommand{\ZH}{\mathbb H}

\newcommand{\ch}{\mathcal H}
\newcommand{\cp}{\mathcal P}
\newcommand{\Z}{\mathbb Z}
\newcommand{\ZS}{\mathbb S}

\newcommand{\si}{\sigma}
\newcommand{\be}{\beta}
\newcommand{\Si}{\Sigma}
\newcommand{\N}{\mathbb N}

\newcommand{\supp}{\operatorname{supp}}

\newcommand{\beq}{\begin{equation}}
	\newcommand{\eeq}{\end{equation}}
\newcommand{\beqq}{\begin{equation*}}
	\newcommand{\eeqq}{\end{equation*}}
\newcommand{\ben}{\begin{eqnarray}}
	\newcommand{\een}{\end{eqnarray}}
\newcommand{\beno}{\begin{eqnarray*}}
	\newcommand{\eeno}{\end{eqnarray*}}

\def\d{\delta}
\def\de{\delta}
\def\la{\lambda}

\def\e{\varepsilon}

\begin{document}

\title[Sharp microlocal Kakeya--Nikodym estimates]{Sharp microlocal Kakeya--Nikodym estimates for eigenfunctions with applications}

 \author{Chuanwei Gao}
	\address{School of Mathematical Sciences, Capital Normal University, Beijing 100048, China}
	\email{cwgao@cnu.edu.cn}
	\author{Shukun Wu}
	\address{Department of Mathematics, Indiana University Bloomington, USA}
	\email{shukwu@iu.edu}
	\author{Yakun Xi}
	\address{School of Mathematical Sciences, Zhejiang University, Hangzhou 310027, China}
	\email{yakunxi@zju.edu.cn}

\begin{abstract}
We extend the microlocal Kakeya--Nikodym bounds for eigenfunctions of Blair--Sogge to a larger range of exponents, which is optimal in all dimensions $n\ge3$ on general manifolds. On manifolds of constant sectional curvature, we introduce a new anisotropic variant of the microlocal Kakeya--Nikodym norm that further enlarges the admissible $p$-range. As a corollary, by combining our results with a recent theorem of Hou, we obtain improved $L^p$ bounds for Hecke--Maass forms on compact hyperbolic $3$-manifolds. 

{In particular, our method applies to general H\"ormander operators, and we characterize the $L^q \to L^p$ boundedness of H\"ormander operators with positive-definite phase in {all} dimensions $n\ge3$, thereby fully resolving a question going back to H\"ormander.} 
Further applications include improved $L^q \to L^p$ Fourier extension bounds, and improved bounds related to the Bochner--Riesz conjecture in $\mathbb R^3$.
\end{abstract}

\maketitle

\section{Introduction}

\subsection{Kakeya--Nikodym estimates for eigenfunctions of the Laplacian} Let $(M,g)$ be a smooth, compact, boundaryless Riemannian manifold of dimension $n\ge 2$. 
Let $\Delta_g$ denote the Laplace--Beltrami operator on $M$, and consider $L^2$-normalized eigenfunctions $e_\lambda$ satisfying
\[
    -\Delta_g e_\lambda=\lambda^2 e_\lambda,
\]
where $\lambda\ge 0$. After possibly rescaling the metric, we may assume that the injectivity radius of $(M,g)$ is at least $10$. 

Sogge's celebrated estimates \cite{sogge1988concerning} provide sharp bounds for the $L^p$ norms of eigenfunctions in terms of the eigenvalue $\lambda$:
\begin{equation}\label{eq:sogge}
    \|e_\lambda\|_{L^p(M)} \lesssim \lambda^{\sigma(p)}\|e_\lambda\|_{L^2(M)}.
\end{equation}
Here
\[
    \sigma(p)=
    \begin{cases}
        \dfrac{n-1}{2} \left(\dfrac{1}{2}-\dfrac{1}{p}\right), & \text{if } 2\le p\le 2\,\dfrac{n+1}{n-1}, \\[8pt]
        n \left(\dfrac{1}{2}-\dfrac{1}{p}\right)-\dfrac{1}{2}, & \text{if } 2\,\dfrac{n+1}{n-1}\le p\le \infty.
    \end{cases}
\]
These bounds describe how much eigenfunctions can concentrate as $p$ increases.
They are sharp on the standard sphere $\mathbb S^n$:
zonal harmonics saturate the upper bound for large $p$, while highest-weight spherical harmonics saturate the upper bound for small $p$.

It is of great interest to study the concentration of eigenfunctions on general Riemannian manifolds. Since $L^p$ norms quantify concentration, a natural question is whether Sogge's estimates can be improved on manifolds whose geometry differs from that of the sphere. As spherical harmonic examples suggest, for an eigenfunction to have large $L^p$ norms for some $2<p<2\,\tfrac{n+1}{n-1}$, it should concentrate in a $\lambda^{-1/2}$-neighborhood of a geodesic. The Kakeya--Nikodym norm is designed to measure such concentration. For $f\in L^2(M)$, define
\begin{equation}
\label{BS-microlocal-infor}
\|f\|_{\KN(\lambda)} := \Big( \sup_{\gamma\in\Pi} \lambda^{\frac{n-1}{2}} \int_{T_{\lambda^{-1/2}}(\gamma)} |f|^2 \, dx \Big)^{1/2},
\end{equation}
where $T_{\lambda^{-1/2}}(\gamma)$ is the $\lambda^{-1/2}$-tubular neighborhood of a unit-length geodesic segment $\gamma$ and $\Pi$ is the collection of all such segments. 
Building on \cite{bourgain2009geodesic,sogge2011kakeya}, Blair and Sogge \cite{blair2015refined} proved that, for $n=2$ and every $\varepsilon>0$ and $\lambda\ge1$,
\begin{equation}\label{eq:BSKN2A}
\|e_\lambda\|_{L^4(M)}\lesssim_\varepsilon \lambda^{\varepsilon}\,
    \|e_\lambda\|_{L^2(M)}^{1/2}\,\|e_\lambda\|_{\KN(\lambda)}^{1/2}.
\end{equation}
This estimate is sharp outside of the $\lambda^\varepsilon$ loss, being saturated by highest-weight spherical harmonics on the sphere. 
In fact, Blair and Sogge introduced a \emph{microlocal} Kakeya--Nikodym norm $\|\,\cdot\,\|_{\MKN(\lambda)}$, defined via a wave-packet decomposition adapted to the geodesic flow. We postpone the precise definition to Section \ref{sec wave}. Roughly speaking, the construction is as follows. 

Fix $\varepsilon_0=\varepsilon_0(\varepsilon)\in(0,1)$ and a dyadic $s$ with $\lambda^{-1/2+\varepsilon_0}\le s\le 1$. Given an eigenfunction $e_\lambda$, take a wave-packet decomposition at each scale $s$,
\[
e_\lambda=\sum_{T^s} e_{\lambda,T^s},
\]
where each $e_{\lambda,T^s}$ is microlocalized to an $s$-tubular neighborhood $T^s\subset M$ of a unit-length geodesic segment in $(M,g)$. The microlocal Kakeya--Nikodym norm of $e_\lambda$ is
\[
\|e_\lambda\|_{\MKN(\lambda)}
:=\sup_{\lambda^{-1/2+\varepsilon_0}\le s\le 1}\ \sup_{T^s}\ s^{-\frac{n-1}{2}}\ \|e_{\lambda,T^s}\|_{L^2(M)}.
\]

With this microlocal norm, Blair and Sogge proved the stronger bound
\begin{equation}\label{eq:BSMKN2d}
    \|e_\lambda\|_{L^4(M)}\lesssim_\varepsilon \lambda^{\varepsilon}\,
    \|e_\lambda\|_{L^2(M)}^{1/2}\,\|e_\lambda\|_{\MKN(\lambda)}^{1/2}.
\end{equation}
Blair--Sogge \cite{blair2017refined} later proved that, in all dimensions $n\ge2$, if $2\,\frac{n+2}{n}<p<2\,\frac{n+1}{n-1}$, then for any $\lambda\ge1$,
\begin{equation}\label{eq:bshighed}
\|e_\lambda\|_{L^p(M)}\lesssim
\lambda^{\frac{n-1}{2}-\frac{n}{p}}\,
\|e_\lambda\|_{L^2(M)}^{2-\frac{2(n+1)}{p(n-1)}}\,\|e_\lambda\|_{\KN(\lambda)}^{\frac{2(n+1)}{p(n-1)}-1}.
\end{equation}
{As observed by Blair and Sogge, enlarging the range of exponents $p$ for which \eqref{eq:bshighed} remains valid is a subtle problem. Any further progress in this direction demands a finer analysis than what is attainable through the bilinear estimates \cite{lee2006linear,tao2003sharp}.
} As one main result of this paper, we extend their range in higher dimensions and, in the constant sectional curvature setting, obtain even stronger bounds.

\begin{theorem}[Kakeya--Nikodym estimates for eigenfunctions]\label{theo main4}
Let $n\ge 3$. Then, for any $\varepsilon>0$ and $\lambda\ge 1$,
\begin{equation}\label{eq:bs}
\|e_\lambda\|_{L^p(M)}\lesssim_\varepsilon
    \lambda^{\frac{n-1}{2}-\frac{n}{p}+\varepsilon}\,
    \|e_\lambda\|_{L^2(M)}^{2-\frac{2(n+1)}{p(n-1)}}\,\|e_\lambda\|_{\KN(\lambda)}^{\frac{2(n+1)}{p(n-1)}-1},
\end{equation}
for all
\[
2\,\frac{3n+1}{3n-3}\le p\le 2\,\frac{n+1}{n-1}.
\]
Furthermore, if $M$ has constant sectional curvature, then \eqref{eq:bs} holds for the larger range
\[
   2\,\frac{3n+2}{3n-2}\le p\le 2\,\frac{n+1}{n-1}.
\]
\end{theorem}

As is customary in the study of Laplace eigenfunctions, we prove Theorem \ref{theo main4} by establishing stronger bounds for associated approximate spectral projectors. Fix $\chi\in\mathcal S(\mathbb R)$ with $\chi(0)=1$ and $\operatorname{supp}\widehat\chi\subset\big(\tfrac12,1\big)$. For $f\in L^2(M)$, let $\{e_{\lambda_j}\}_j$ be an $L^2(M)$-orthonormal basis of eigenfunctions, and define
\[
\chi_\lambda f:=\chi\big(\lambda-\sqrt{-\Delta_g}\big)f=\sum_j \chi(\lambda-\lambda_j)\,\langle f,e_{\lambda_j}\rangle\,e_{\lambda_j}.
\]
Then $\chi_\lambda$ reproduces $e_\lambda$ in the sense that $e_\lambda=\chi_\lambda e_\lambda$.
By standard reductions \cite[Lemma 5.1.3]{sogge2017fourier}, modulo rapidly decaying errors in $\lambda$, $\chi_\lambda$ can be written as an oscillatory integral operator with a positive-definite $n\times n$ Carleson--Sj\"olin phase:
\begin{equation}
\label{tilde-chi}
    \widetilde\chi_\lambda f(x)
=\lambda^{\frac{n-1}{2}}\int e^{i\lambda\,d_g(x,y)}\,a_\lambda(x,y)\,f(y)\,dy.
\end{equation}
Here $d_g(x,y)$ is the Riemannian distance on $M$, and $a_\lambda(x,y)$ is smooth. The phase $d_g(x,y)$ satisfies an $n\times n$ variant of the positive-definite Carleson--Sj\"olin conditions \cite{sogge2017fourier}. 

{We analyze \eqref{tilde-chi} using a wave-packet decomposition aligned with the geodesic flow and a multiscale analysis inspired by recent advances in Fourier restriction.  In particular, we combine the recently developed refined decoupling techniques with incidence estimates to address this problem. This yields microlocal Kakeya--Nikodym bounds for $\chi_\lambda$, which in turn give the first assertion of Theorem \ref{theo main4}.}

On constant curvature manifolds, however, the isotropic microlocal norm $\|\cdot\|_{\MKN(\lambda)}$ is not strong enough to register the anisotropic concentration that can occur. We therefore introduce a new anisotropic microlocal Kakeya--Nikodym norm. When $M$ has constant sectional curvature, there exist local bi-Lipschitz diffeomorphisms that straighten geodesics to lines and totally geodesic submanifolds to affine subspaces. See \cite{gao2025curved}. In these charts, we can naturally define rectangular boxes as in the Euclidean case, and the following norm is unambiguous:
\[
\|e_\lambda\|_{\widetilde{\MKN}(\lambda)}:=\sup_{\vec s\in[\lambda^{-1/2+\varepsilon_0},1]^{n-1}}\ \sup_{\Box_{\vec s}}\ \prod_{i=1}^{n-1}s_i^{-1/2}\,\|e_{\lambda,\Box_{\vec s}}\|_{L^2(M)},
\]
where $e_{\lambda,\Box_{\vec s}}$ ranges over wave envelopes supported in boxes of dimensions $1\times s_1\times\cdots\times s_{n-1}$, with each $s_i\in[\lambda^{-1/2+\varepsilon_0},1],\,i=1,2,\ldots,n-1$. This anisotropic norm is designed to exclude dense packing near totally geodesic submanifolds. We postpone the precise definition to Section \ref{sec wave}. Our bounds for $\chi_\lambda$ are as follows.

\begin{theorem}[Kakeya--Nikodym estimates for spectral projector]\label{theo main5}
    For $2\,\frac{3n+1}{3n-3}\le p\le 2\,\frac{n+1}{n-1}$, any $\varepsilon>0$ and $\lambda\ge 1$, 
\begin{equation}\label{eq:micc}
    \|\chi_\lambda f\|_{L^p(M)}\lesssim_\varepsilon  \lambda^{\frac{n-1}{2}-\frac{n}{p}+\varepsilon}\,\|f\|_{L^2(M)}^{2-\frac{2(n+1)}{p(n-1)}}\,
    \|f\|_{\MKN(\lambda)}^{\frac{2(n+1)}{p(n-1)}-1}.
\end{equation}
Furthermore, if $M$ has constant sectional curvature, then for any $\varepsilon>0$ and $\lambda\ge 1$,
\begin{equation}\label{eq:mica}
    \|\chi_\lambda f\|_{L^p(M)}\lesssim_\varepsilon  \lambda^{\frac{n-1}{2}-\frac{n}{p}+\varepsilon}\,\|f\|_{L^2(M)}^{2-\frac{2(n+1)}{p(n-1)}}\,
    \|f\|_{\widetilde{\MKN}(\lambda)}^{\frac{2(n+1)}{p(n-1)}-1},
\end{equation}
    for $ 2\,\frac{3n+2}{3n-2}\le p\le 2\,\frac{n+1}{n-1}$.
\end{theorem}
We remark that, using standard level set analysis, Theorem \ref{theo main5} also yields sharp $L^q\to L^p$ estimates for the operator $\chi_\lambda$. 
However, such $L^q\to L^p$ bounds alone do not provide useful information for eigenfunctions, since one needs $q=2$ to compare $\|e_\lambda\|_{L^p}$ with $\|e_\lambda\|_{L^2}$. For instance, in dimension $2$, the endpoint $(4,4)$ estimate gives
$\|e_\lambda\|_{L^4}\lesssim_\varepsilon \lambda^\varepsilon \|e_\lambda\|_{L^4}$, which is trivial for eigenfunctions. Thus the microlocal Kakeya--Nikodym estimates strengthen the $L^q\to L^p$ theory: they encode microlocal concentration and supply spectral--geometric information. {Moreover, on general manifolds, 
\eqref{eq:micc} cannot hold beyond the range  $2\,\frac{3n+1}{3n-3}\le p\le 2\,\frac{n+1}{n-1}$,  see Remark \ref{rem MKN implies p,q}.}

The constant sectional curvature assumption in Theorem \ref{theo main5} is natural in several senses. 
First, it is required for the definition of the anisotropic Kakeya--Nikodym norm, and local incidence bounds in $\mathbb R^n$ transfer verbatim to $(M,g)$. 
Second, in the flat case $d_g(x,y)=|x-y|$, the estimate \eqref{eq:mica} implies, via standard reductions, sharp $L^p$ bounds for the Bochner--Riesz multipliers. See Corollary \ref{cor bochner} below. 
Third, manifolds with constant sectional curvature are a core class in the study of Laplace eigenfunctions, which leads to a direct application of our results described in the next subsection.

\subsection{Improved $L^p$ bounds for Hecke--Maass forms} As mentioned earlier, Kakeya--Nikodym estimates for eigenfunctions have been widely used to obtain improved $L^p(M)$ estimates for eigenfunctions, especially for manifolds with negative sectional curvature; see a series of papers by Blair, Huang and Sogge \cite{blair2018concerning,blair2019logarithmic,blair2024improved,huang2025curvature}.

From a broader viewpoint, understanding eigenfunction concentration lies at the heart of quantum chaos. The quantum ergodicity theorem of Shnirel'man \cite{shnirel1974ergodic}, Colin de Verdi\`ere \cite{colin1985ergodicite}, and Zelditch \cite{zelditch1987uniform} shows that on manifolds whose geodesic flow is ergodic, almost all high-frequency eigenfunctions equidistribute in the weak limit, while the Quantum Unique Ergodicity (QUE) conjecture of Rudnick--Sarnak \cite{rudnick1994behaviour} predicts equidistribution of the entire sequence on negatively curved manifolds. Improving $L^p$ bounds over Sogge's universal estimates provides another quantitative way to measure equidistribution.

Nonetheless, the gains for the Kakeya--Nikodym norms obtained by Blair, Huang, and Sogge involve only a negative power of $\log\lambda$, limiting their strength and precluding application of our Theorem \ref{theo main4} due to the $\lambda^\varepsilon$-loss. In the arithmetic setting, however, one often expects stronger equidistribution results. For instance, QUE has been proved for \emph{Hecke--Maass forms} (i.e.\ joint eigenfunctions of the Laplacian and the Hecke operators) by Lindenstrauss \cite{lindenstrauss2006invariant}. In this setting, $L^p$ norms of Hecke--Maass forms are deeply linked to central values of $L$-functions; indeed, sufficiently strong uniform control of these $L^p$ norms would imply instances of the Lindel\"of hypothesis \cite{sarnak1993arithmetic}. By the amplification method of Iwaniec and Sarnak \cite{iwaniec1995norms}, one can obtain polynomial improvements in this setting; see also \cite{marshall2016geodesic}. For a three-dimensional compact arithmetic congruence hyperbolic manifold, Hou \cite{hou2024restrictions} obtained power-improved Kakeya--Nikodym norms for Hecke--Maass forms.

To describe Hou's result, we set
\[
M = \Gamma \backslash \mathbb{H}^3,
\]
where $\mathbb{H}^3=\mathrm{SL}(2,\mathbb{C})/\mathrm{SU}(2)$ is hyperbolic $3$-space and $\Gamma\subset\mathrm{SL}(2,\mathbb{C})$ is a cocompact arithmetic congruence lattice arising from a quaternion division algebra over a number field with one complex place. Equipped with the induced metric $g$, $(M,g)$ is a compact hyperbolic $3$-manifold. Let $\psi_\lambda$ be an $L^2$-normalized Hecke--Maass form on $M$, that is, an eigenfunction of the Laplace--Beltrami operator $\Delta_g$ and of the unramified Hecke operators. Then
\[
-\Delta_g\psi_\lambda(x)=\lambda^2\psi_\lambda(x).
\]
Hou \cite{hou2024restrictions} proved that, for any $\varepsilon>0$ and $\lambda\ge1$,
\begin{equation}\label{eq:Hou}
\sup_{\gamma\in\Pi}\|\psi_\lambda\|_{L^2(T_{\lambda^{-1/2}}(\gamma))}
\lesssim_\varepsilon\lambda^{-\frac1{20}+\varepsilon}\|\psi_\lambda\|_{L^2(M)}.
\end{equation}

Hou then applies Blair--Sogge's estimate \eqref{eq:bshighed} to obtain
\begin{equation}
\|\psi_\lambda\|_{L^p(M)}\lesssim_\varepsilon\,\lambda^{\frac12-\frac1p-\frac1{20}\bigl(\tfrac4p-1\bigr)+\varepsilon}
= \lambda^{\frac{11}{20}-\frac{6}{5p}+\varepsilon}\|\psi_\lambda\|_{L^2(M)},
\qquad 10/3\le p<4.
\end{equation}
Hou's result also yields improved $L^p$ bounds for $2<p<10/3$ via interpolation:
\begin{equation}\label{eq:Hou2}
\|\psi_\lambda\|_{L^p(M)}\lesssim_\varepsilon\,\lambda^{\frac{19}{40}-\frac{19}{20p}+\varepsilon}\|\psi_\lambda\|_{L^2(M)},
\qquad 2<p<10/3.
\end{equation}
Since our results sharpen those of Blair and Sogge in the constant-curvature setting, it is straightforward to combine Theorems \ref{theo main4} and \eqref{eq:Hou} to further improve \eqref{eq:Hou2} for all $2<p<\tfrac{10}{3}$.

\begin{corollary}[Improved $L^p$ bounds for Hecke--Maass forms]\label{theo main6}
For any $\varepsilon>0$ and $\lambda\ge1$,
\[
\|\psi_\lambda\|_{L^p(M)}\lesssim_\varepsilon\,\lambda^{\frac{11}{20}-\frac{6}{5p}+\varepsilon}\|\psi_\lambda\|_{L^2(M)},
\qquad 22/7\le p<4,
\]
and, by interpolation,
\[
\|\psi_\lambda\|_{L^p(M)}\lesssim_\varepsilon\,\lambda^{\frac{37}{80}-\frac{37}{40p}+\varepsilon}\|\psi_\lambda\|_{L^2(M)},
\qquad 2<p<22/7.
\]
\end{corollary}

The proofs of our main theorems build on recent advances in Fourier restriction theory. 
These applications underscore the potential interplay between restriction estimates and the analysis of Hecke--Maass forms.

\medskip

\subsection{H\"ormander oscillatory integral operators} Our methods apply to  H\"ormander oscillatory integral operators. We describe some direct consequences in this section. 
Let $n\ge 2$.  We write $B_r^{m}$ for the $m$-dimensional Euclidean ball of radius $r$ centered at the origin. Let $a\in C_c^\infty(\R^n\times\R^{n-1})$ be a nonnegative amplitude with $\supp a\subset B_1^n\times B_1^{n-1}$.
 Let $\phi\colon B_1^n\times B_1^{n-1}\to\R$ be smooth and satisfy the Carleson--Sj\"olin conditions
\begin{itemize}
  \item[\H] $\rank \partial^2_{\xi x}\phi(x,\xi)=n-1$ for all $(x,\xi)\in B_1^n\times B_1^{n-1}$;
  \item[\HH] Writing
  \[
    G_0(x,\xi):=\bigwedge_{j=1}^{n-1}\partial_{\xi_j}\partial_x\phi(x,\xi),\qquad
    G(x,\xi):=\frac{G_0(x,\xi)}{|G_0(x,\xi)|},
  \]
  the curvature condition
  \[
    \det\partial^2_{\xi\xi}\big\langle \partial_x\phi(x,\xi),\,G(x,\xi_0)\big\rangle\big|_{\xi=\xi_0}\neq 0
  \]
  holds for all $(x,\xi_0)\in\supp a$.
\end{itemize}

For $\lambda\ge1$ we consider the rescaled operator
\begin{equation}\label{eq:00}
  \mathcal H^\lambda f(x):=\int_{B^{n-1}_1} e^{2\pi i\,\phi^\lambda(x,\xi)}\,a^\lambda(x,\xi)\,f(\xi)\,d\xi,
\end{equation}
where $a^\lambda(x,\xi):=a(x/\lambda,\xi)$ and $\phi^\lambda(x,\xi):=\lambda\,\phi(x/\lambda,\xi)$. 
We call $\mathcal H^\lambda$ a {\it H\"ormander operator} if $\phi$ satisfies \H and \HH. Throughout the paper, the phase $\phi$ and amplitude $a$ are considered fixed; implicit constants in estimates may depend on these data.

\smallskip 

A fundamental example is the Fourier extension operator
\begin{equation}\label{eq:02a}
  E_\Sigma f(x):=\int_{B_1^{n-1}} e^{2\pi i\,(x'\cdot\xi+x_n\psi(\xi))}\,f(\xi)\,d\xi,
\end{equation}
where $\psi$ satisfies
\[
  \rank\Big(\frac{\partial^2\psi}{\partial \xi_i\partial \xi_j}\Big)_{(n-1)\times(n-1)}=n-1.
\]
As $\psi$ parametrizes the smooth hypersurface $\Sigma=\{(\xi,\psi(\xi)):\ \xi\in B^{n-1}_1\}\subset\R^n$ with everywhere nonzero Gaussian curvature, $E_\Sigma$ is the Fourier extension (adjoint restriction) operator associated to $\Sigma$.
Historically, there has been strong interest in establishing $(q,p)$ bounds of the form
\begin{equation}\label{eq:localEfSimple}
  \|E_\Sigma f\|_{L^p(B^n_\lambda)}\lesssim_\varepsilon\lambda^{\varepsilon}\,\|f\|_{L^q(B^{n-1}_1)}.
\end{equation}
Subject to an $\varepsilon$-removal argument by Tao \cite{Tao-BR-Restriction}, Stein's restriction conjecture predicts the sharp $(q,p)$ range for which \eqref{eq:localEfSimple} holds.
\begin{conjecture}[Stein \cite{Stein-restriction}]\label{conj:steinEf}
Let $n\ge 2$ and $p\geq\tfrac{2n}{n-1}$. If $q$ satisfies $q'\le \tfrac{n-1}{n+1}\,p$, where $q'$ is the H\"older conjugate of $q$, then
\begin{equation}\label{eq:steinEf}
\text{for every $\varepsilon>0$ and $\lambda\ge 1$,}\quad
\|E_\Sigma f\|_{L^p(B^n_\lambda)}\lesssim_\varepsilon \lambda^{\varepsilon}\,\|f\|_{L^q(B^{n-1}_1)}.
\end{equation}
\end{conjecture}

\smallskip 

H\"ormander operators provide a variable-coefficient framework that underlies both the Fourier restriction phenomena and many other problems in harmonic analysis. 
Motivated by the connection to Euclidean Fourier analysis, H\"ormander \cite{hormander1973oscillatory} initiated the program of determining the optimal $(q,p)$ range for which the following estimate is true:
\begin{equation}\label{eq Hor}
  \text{for every $\varepsilon>0$ and $\lambda\ge1$,}\quad
  \|\mathcal H^\lambda f\|_{L^p(B^n_\lambda)}\lesssim_\varepsilon \lambda^{\varepsilon}\,\|f\|_{L^q(B_1^{n-1})}.
\end{equation}

We shall refer to \eqref{eq Hor} as the \emph{$(q,p)$ H\"ormander estimate} for $\mathcal H^\lambda$.
Originally, H\"ormander conjectured that \eqref{eq Hor} holds in the same $(q,p)$ range as stated in Conjecture \ref{conj:steinEf}.
If true, this would simultaneously imply both Stein's Fourier restriction conjecture and the Bochner--Riesz conjecture. 
H\"ormander proved this result when $n=2$, and in higher dimensions $n\ge3$, Stein \cite{stein1986oscillatory} established \eqref{eq Hor} in a smaller range $p\ge 2\,\frac{n+1}{n-1}$ with $q'\le \tfrac{n-1}{n+1}\,p$ (with no $\lambda^\varepsilon$ loss).
However, Bourgain \cite{bourgain1991lp} later constructed counterexamples to H\"ormander's conjecture, exhibiting the \emph{Kakeya compression phenomenon} and showing that \eqref{eq Hor} can only be true when
\begin{equation}\label{eq:Hor range1}
  p\ge \begin{cases}
    2\,\dfrac{n+1}{n-1}, & n \text{ odd},\\[6pt]
    2\,\dfrac{n+2}{n},   & n \text{ even}.
  \end{cases}
\end{equation}
This shows that Stein's result in \cite{stein1986oscillatory} is sharp in odd dimensions.
Subsequent work has largely focused on $(p,p)$ H\"ormander estimates and on extending the admissible range of $p$. On this front, for even $n$, sharp results up to the endpoint were obtained by Bourgain--Guth \cite{bourgain2011bounds}.

In this paper, we focus on the case where the phase is \emph{positive-definite} in the sense of Lee \cite{lee2006linear}:
\begin{itemize}
  \item[\HHH] For every $(x,\xi_0)\in\supp a$, all eigenvalues of
  \[
    \partial^2_{\xi\xi}\big\langle \partial_x\phi(x,\xi),\,G(x,\xi_0)\big\rangle\big|_{\xi=\xi_0}
  \]
  are strictly positive.
\end{itemize}
In the setting of Conjecture \ref{conj:steinEf}, this is the elliptic case: the hypersurface $\Sigma$ has all principal curvatures of the same sign (equivalently, its second fundamental form is definite on $\supp a$).

A key distinction between H\"ormander operators \eqref{eq:00} and Fourier extension operators \eqref{eq:02a} is the Kakeya compression phenomenon: the main contribution of $\mathcal H^\lambda f$ may concentrate in a small neighborhood of a lower-dimensional submanifold, with different manifestations in odd and even dimensions. 
Under the positive-definite hypothesis \HHH, transverse equidistribution forces any such concentration to occur in a  $\lambda^{1/2}$-neighborhood of the submanifold.
For general phases satisfying only the weaker condition \HH, however, it can compress further to a unit-scale neighborhood.
In contrast, no analogous compression is expected for the extension operator \eqref{eq:02a} if the Kakeya maximal conjecture is true. 
This mechanism gives three distinct optimal ranges of $p$. 
Assuming that $\ch^\la$ satisfies \H and \HHH, Guth, Hickman, and Iliopoulou \cite{GHI} proved $(p,p)$ H\"ormander estimates in the optimal range
\begin{equation}\label{eq:Hor range2}
  p\ge
  \begin{cases}
    2\,\dfrac{3n+1}{3n-3}, & n \text{ odd},\\[6pt]
    2\,\dfrac{3n+2}{3n-2}, & n \text{ even}.
  \end{cases}
\end{equation}
The proof in \cite{GHI}, broadly speaking, relies on a dichotomy based on the $k$-broadness of the operator $\mathcal{H}^\lambda$.
When $k$ is large, a sharp $k$-broad estimate is applied to directly obtain the $(p,p)$-estimate in the desired range \eqref{eq:Hor range2}.
When $k$ is small, the argument uses Bourgain--Demeter decoupling theorem along with the parabolic rescaling to reduce the problem to a finer scale, and completes the proof by induction.
In particular, the loss from using the Bourgain--Demeter decoupling theorem can be compensated by a gain from the parabolic rescaling.
However, to the best of our knowledge, this approach breaks down when seeking off-diagonal $(q,p)$ estimates, primarily because parabolic rescaling does not yield a compensating gain in this setting, while the decoupling theorem still incurs a critical loss.

\smallskip

{Adapting the proof of the first part of Theorem \ref{theo main5} to H\"ormander operators, and combining this with the diagonal estimate of \cite{GHI} in even dimensions, we obtain the sharp range of $(q,p)$ estimates in the positive-definite case.

\begin{theorem}[Sharp $(q,p)$ estimates for H\"ormander operators]\label{theo main1}
Let $n\ge3$ and suppose $\mathcal H^\lambda$ satisfies \H and \HHH. Then the $(q,p)$ H\"ormander estimate holds if and only if one of the following holds:
\begin{itemize}
    \item[\rm (i)] $n$ is odd, 
    \[
    p\ge 2\,\frac{3n+1}{3n-3}
    \qquad\text{and}\qquad
    q'\le \frac{n-1}{n+1}p.
    \]

    \item[\rm (ii)] $n$ is even, 
    \[
    p\ge 2\,\frac{3n+2}{3n-2}
    \]
    and
    \[
    \frac1q\le \min\biggl\{\frac{3n-2}{4}-\frac{3n}{2p},\,
    1-\frac{n+1}{(n-1)p}\biggr\}.
    \]
    Equivalently, in even dimensions this is the union of the two regions
    \[
    2\,\frac{3n+2}{3n-2}\le p\le 2\,\frac{3n+1}{3n-3},
    \qquad
    \frac1q\le \frac{3n-2}{4}-\frac{3n}{2p},
    \]
    and
    \[
    p\ge 2\,\frac{3n+1}{3n-3},
    \qquad
    q'\le \frac{n-1}{n+1}p.
    \]
\end{itemize}
\end{theorem}
Indeed, we prove  $(q,p)$ H\"ormander estimate in all dimensions under the conditions
\[
p\ge 2\,\frac{3n+1}{3n-3}
\qquad\text{and}\qquad
q'\le \frac{n-1}{n+1}p.
\]
In even dimensions, the additional range of admissible pairs $(q,p)$ is then obtained by interpolating between our endpoint estimate and the $(p,p)$ endpoint result of \cite{GHI}, which gives the range
    \[
    2\,\frac{3n+2}{3n-2}\le p\le 2\,\frac{3n+1}{3n-3},
    \qquad
    \frac1q\le \frac{3n-2}{4}-\frac{3n}{2p},
    \]
What is especially interesting is that we also show this additional condition is necessary (see Section \ref{sec example}), thereby completing the picture. See Figure \ref{fig 1}}

\begin{remark}
\rm
A similar additional condition arises in the non-elliptic case. 
In Remark \ref{bourgain-guth-example}, we show that for H\"ormander operators $\mathcal H^\lambda$ satisfying only \H\ and \HH, the range obtained by combining the Bourgain--Guth bound with the Tomas--Stein theorem indeed yields a complete solution in this setting.
\end{remark}

\begin{figure}[ht]
\centering
\begin{tikzpicture}[x=13cm,y=6.8cm]

\draw[->, thin] (0,0) -- (0,1.08) node[left] {$\frac1q$};
\draw[->, thin] (0,0) -- (0.90,0) node[below] {$\frac1p$};

\coordinate (A) at (0,0.98);
\coordinate (K) at (0.36,0.66);  
\coordinate (E) at (0.68,0.20);   
\coordinate (B) at (0.36,0);
\coordinate (C) at (0.68,0);

\fill[gray!20] (0,0) -- (A) -- (K) -- (B) -- cycle;
\fill[blue!18] (B) -- (K) -- (E) -- (C) -- cycle;

\draw[very thick] (A) -- (K) -- (E) -- (C);
\draw[very thick, gray!70] (K) -- (B);

\fill (A) circle (1.2pt);
\fill (K) circle (1.2pt);
\fill (E) circle (1.2pt);

\draw (B) -- ++(0,-0.014);
\draw (C) -- ++(0,-0.014);

\node[below=4pt] at (B)
{$\frac{3n-3}{2(3n+1)}$};

\node[below=4pt] at (C)
{$\frac{3n-2}{2(3n+2)}$};

\node[above right=1pt] at (A)
{\scriptsize $(0,1)$};

\node[above right=2pt, align=left] at (K)
{ $\left(\frac{3n-3}{2(3n+1)},\,\frac{3n-1}{2(3n+1)}\right)$\\
\scriptsize \qquad  Theorem \ref{theo main1}};

\node[above right=2pt, align=left] at (E)
{ $\left(\frac{3n-2}{2(3n+2)},\,\frac{3n-2}{2(3n+2)}\right)$\\
\scriptsize\quad \ \qquad\cite{GHI}};

\node[scale=1.2] at (0.17,0.34) {odd $n$};

\node[scale=1.05, align=center] at (0.52,0.16)
{additional range\\[1pt] for even $n$};

\end{tikzpicture}
\caption{Admissible regions in the $(\tfrac1p,\tfrac1q)$-plane. 
The gray region is the admissible range in odd dimensions, while the blue region is the additional range available in even dimensions. The figure is schematic and not drawn to scale.}
\label{fig 1}
\end{figure}

As is the case for Theorem \ref{theo main4}, we actually prove stronger microlocal Kakeya--Nikodym  estimates tailored to induction on scales. The  microlocal Kakeya--Nikodym norm can similarly be defined for a given H\"ormander operator. Again, the full definition is somewhat lengthy and is postponed to Section 2. 
Our microlocal Kakeya--Nikodym estimates are as follows.
\begin{theorem}[Microlocal Kakeya--Nikodym estimates for H\"ormander operators]\label{theo main2}
Let $n\ge3$ and let $\mathcal H^\lambda$ satisfy \H and \HHH. For every $\varepsilon>0$ and $\lambda\ge1$, and for all $2\,\frac{3n+1}{3n-3}\le p\le2\,\frac{n+1}{n-1}$,
\begin{equation}\label{eq:mic}
    \|\mathcal H^\lambda f\|_{L^p(B^n_\lambda)}\lesssim_\varepsilon\ \lambda^\varepsilon\, \|f\|_{L^2(B^{n-1}_1)}^{2-\frac{2(n+1)}{p(n-1)}}\,
    \|f\|_{\MKN(\lambda)}^{\frac{2(n+1)}{p(n-1)}-1}.
\end{equation}
\end{theorem}
{\begin{remark}\label{rem MKN implies p,q}
As noted above, Theorem \ref{theo main2} yields Theorem \ref{theo main1} on the scaling line $q'=\tfrac{n-1}{n+1}\,p$. Indeed, observe that 
\[
\|f\|_{\MKN(\lambda)}\lesssim \|f\|_{L^\infty(B_1^{n-1})}.
\]
Taking $f=\mathbf 1_E$ in \eqref{eq:mic} gives the restricted-type $(q,p)$ bound at the endpoint $q'=\tfrac{n-1}{n+1}\,p$. A standard level-set decomposition then upgrades this to the strong $(q,p)$ estimate \eqref{eq Hor} at the endpoint. This endpoint bound then implies the full range $q'\le \tfrac{n-1}{n+1}\,p$ appearing in Theorem \ref{theo main1} by H\"older's inequality on the bounded domain $B^{n-1}_1$.

Noting that Theorem \ref{theo main1} is sharp for all dimensions on the line $q'=\tfrac{n-1}{n+1}\,p$, (see Section \ref{sec example}), the threshold $p\ge 2\,\tfrac{3n+1}{3n-3}$ in Theorem \ref{theo main2} is therefore optimal in all dimensions and cannot be lowered. For the same reason, the first part of Theorem \ref{theo main5} is sharp as well.
\end{remark}}

\begin{remark}
The mixed norm on the right-hand side of \eqref{eq:mic} consists of a smaller $L^2$ norm and a larger microlocal Kakeya--Nikodym norm. One may ask whether the exponent on the larger norm $\|\cdot\|_{\MKN(\lambda)}$ can be lowered. However, as the Knapp example shows, this is impossible.
\end{remark}

\smallskip 

\subsection{Fourier extension and Bochner--Riesz means} Our next application concerns the translation-invariant extension operator $E_\Sigma$ associated with a strictly convex $\Sigma$. In this case stronger bounds than \eqref{eq:mic} are expected. However, as in the second part of Theorem \ref{theo main5}, these improvements cannot be obtained using the same microlocal Kakeya--Nikodym norm. We refine Theorem \ref{theo main2} in the case $\mathcal H^\lambda=E_\Sigma$ using the anisotropic variant $\|{}\cdot{}\|_{\widetilde{\MKN}}$ of the microlocal Kakeya--Nikodym norm.

We briefly describe this norm again in the Fourier extension setting for the convenience of the reader. Fix $\varepsilon_0=\varepsilon_0(\varepsilon)\in(0,1)$. Let $\vec{s}=(s_1,\dots,s_{n-1})\in\mathbb{R}^{n-1}$ with each $s_i$ dyadic and $\lambda^{-1/2+\varepsilon_0}\le s_i\le 1$. At anisotropic scale $\vec s$ take a decomposition
\[
  f=\sum_{\Box_{\vec s}} f_{\Box_{\vec s}},
\]
so that each $E_\Sigma f_{\Box_{\vec s}}$ is microlocalized to a rectangular box in physical space of dimensions
\[
  \lambda \times s_1\lambda \times s_2\lambda \times \cdots \times s_{n-1}\lambda.
\]
Since in this case all wave packets propagate along straight lines, there is no ambiguity in this definition. We then define the anisotropic microlocal Kakeya--Nikodym norm to be the supremum over all $L^2$ averages of such anisotropic packets:
\begin{equation}
\label{anisotropic-KN-norm}
    \|f\|_{\widetilde{\MKN}(\lambda)}
    :=\sup_{\vec s\in[\lambda^{-1/2+\varepsilon_0},1]^{n-1}}\ \sup_{\Box_{\vec s}} \ \prod_{i=1}^{n-1}s_i^{-1/2}\,\|f_{\Box_{\vec s}}\|_{L^2(B^{n-1}_1)}.
\end{equation}
We obtain the following estimates for $E_\Sigma$.

\begin{theorem}[Microlocal Kakeya--Nikodym estimates for extension operators]\label{theo main3}
Let $n\ge 3$, $2\,\tfrac{3n+2}{3n-2}\le p\le2\,\frac{n+1}{n-1}$, and $\Sigma$ strictly convex. Then for every $\varepsilon>0$ and $\lambda\ge1$,
\begin{equation}\label{eq:MKNE}
        \|E_\Sigma f\|_{L^p(B^n_\lambda)}\lesssim_\varepsilon \lambda^\varepsilon \|f\|_{L^2(B_1^{n-1})}^{2-\tfrac{2(n+1)}{p(n-1)}}\,
\|f\|_{\widetilde{\MKN}(\lambda)}^{\tfrac{2(n+1)}{p(n-1)}-1}.
    \end{equation}
\end{theorem}
Similar to Remark \ref{rem MKN implies p,q}, Theorem \ref{theo main3} implies the corresponding $(q,p)$ H\"ormander estimates for the extension operator.

\begin{corollary}[$(q,p)$ estimates for extension operators]\label{cor extension}
    Let $n\ge 3$ and $\Sigma$ strictly convex. Then Conjecture \ref{conj:steinEf} holds for all $p\ge 2\,\tfrac{3n+2}{3n-2}$ and $q'\le \tfrac{n-1}{n+1}\,p$.
\end{corollary}

Corollary \ref{cor extension} is new in all dimensions $n\ge 3$. In particular, in low dimensions, it matches the state-of-the-art $(p,p)$ results for the Fourier extension operator in \cite{Wang-Wu} in the range of $p$ and upgrades them to $(q,p)$ estimates admitting the sharp range of $q$. For $n\ge 4$, progress on the sharp $(q,p)$ restriction problem has largely stagnated. To date, the only effective method is due to \cite{Lee-Rogers-Seeger}, building on the bilinear estimate of \cite{tao2003sharp}. When $n=3$, Shayya \cite{shayya2017weighted} proves \eqref{eq:steinEf} for the smaller range $p\ge 3.25$. It is possible that modifications to our approach could lead to further improvements in high dimensions. See Remark \ref{p-q-rmk}. 

\smallskip 

At this point, it is natural to ask for the maximal range of exponents $p$ for which \eqref{eq:MKNE} holds.
The resolution of the three-dimensional Kakeya set conjecture in \cite{Wang-Zahl} suggests that the only obstruction to thin tubes being essentially disjoint arises when they are densely packed within convex sets.
Among other things, the anisotropic Kakeya--Nikodym norm \eqref{anisotropic-KN-norm} rules out this possibility (see Lemma \ref{KN-L2-lem2}).
This leads us to the following conjecture, which, if true, would imply Conjecture \ref{conj:steinEf} for $n=3$.

\begin{conjecture}\label{conj MKNE} 
Let $n=3$ and $\Sigma$ be strictly convex, then \eqref{eq:MKNE} holds for all $p\ge 3$.
\end{conjecture}

\begin{remark}
\label{algebraic-rmk}
   Geometry becomes more intricate in higher dimensions.
In particular, being densely packed within convex sets is no longer the only obstruction to thin tubes being essentially disjoint---they can also be easily arranged within a thin neighborhood of a real algebraic set.
See, for example, \cite[Section 0.5]{Wang-Wu}.
As a result, the current anisotropic Kakeya--Nikodym norm \eqref{anisotropic-KN-norm} is likely not optimal for studying higher-dimensional microlocal Kakeya--Nikodym estimates for extension operators.
Instead, one probably needs a variant of \eqref{anisotropic-KN-norm} in which the supremum is taken over all wave-packets contained within neighborhoods of low-degree real algebraic sets.
In this case, one expects the corresponding Lemma \ref{KN-L2-lem2} will help enforce a polynomial Wolff axiom (see \cite[Definition 0.14]{Wang-Wu}) on the set of tubes.

\end{remark}

Note that, by setting $s_1=\cdots=s_{n-1}$, the $\widetilde{\MKN}$ norm reduces to the usual microlocal Kakeya--Nikodym norm, and \eqref{eq:MKNE} reduces to
\begin{equation}\label{eq:MKNE'}
        \|E_\Sigma f\|_{L^p(B_\lambda^n)}\lesssim_\varepsilon \lambda^\varepsilon \|f\|_{L^2(B_1^{n-1})}^{2-\tfrac{2(n+1)}{p(n-1)}}\,
\|f\|_{\MKN(\lambda)}^{\tfrac{2(n+1)}{p(n-1)}-1}.
\end{equation}
We will give an example showing that \eqref{eq:MKNE'} cannot hold beyond the range \eqref{eq:Hor range2}, thereby demonstrating the necessity of the anisotropic $\widetilde{\MKN}$ norm. 
See Example \ref{ex1} for details.

\medskip 

Finally, as a last application, Theorem \ref{theo main5} yields an improved range for the Bochner--Riesz multipliers. In the flat case $d_g(x,y)=|x-y|$, the spectral projector bound \eqref{eq:mica} implies, via Remark \ref{rem MKN implies p,q} and standard reductions (see \cite{carleson1972oscillatory,hormander1973oscillatory,Stein-real-methods,GHI}), sharp $L^p$ bounds for the Bochner--Riesz operator $m^\delta(D)$, defined by
\[
m^\delta(D)f(x)=\int_{\mathbb R^n} e^{i\langle x,\xi\rangle}\,m^\delta(\xi)\,\widehat f(\xi)\,d\xi,
\qquad
m^\delta(\xi)=(1-|\xi|^2)_+^\delta.
\]
\begin{corollary}[Improved Bochner--Riesz bounds in $\mathbb R^3$]\label{cor bochner}
The Bochner--Riesz operator $m^\delta(D)$ is bounded on $L^p(\mathbb R^3)$ in the optimal range 
\[
\delta>3\,\Bigl|\frac{1}{2}-\frac{1}{p}\Bigr|-\frac{1}{2}
\]
whenever $\max\{p,p'\}\ge 22/7$, where $p'$ is the H\"older conjugate of $p$.
\end{corollary}

\medskip

\subsection{Discussion and remarks}\label{subsec:remark}
We give further discussion of the microlocal Kakeya--Nikodym norm in both its isotropic and anisotropic forms, and provide several additional remarks along the way.

As noted above, the isotropic microlocal Kakeya--Nikodym norm was introduced by Blair and Sogge \cite{blair2015refined} in the study of spectral projectors. 
Around the same time, Guth \cite{guth2016restriction} introduced a related mixed norm in the Fourier restriction setting, which, in our notation, may be written as
\begin{equation*}
    \|f\|_2^2\,\sup_{\theta}\|f_\theta\|_{L^2_{\rm avg}}^{\,p-2},
\end{equation*}
where $\theta$ ranges over frequency caps of radius $\lambda^{-1/2}$, and
\begin{equation}\label{eq:L2avg}
      \|f_\theta\|_{L^2_{\rm avg}}
    :=\Big(\frac{1}{|\theta|}\int_{\theta}|f_\theta(\xi)|^2\,d\xi\Big)^{1/2},
    \qquad f_\theta=f\,\mathbf 1_\theta.
\end{equation}
Inspired by Guth's mixed norm, Gan and the second author, in their work on local smoothing estimates for wave equations \cite{gan2025local}, introduced the concept of ``wave-packet density,'' which turns out to be closely related to the isotropic microlocal Kakeya--Nikodym norm. 

The $L^2$ average $\sup_{\theta}\|f_\theta\|_{L^2_{\rm avg}}$ works well in the Fourier restriction problem.
In particular, it serves as an effective substitute for the $L^\infty$ norm and behaves well under induction on scales.
However, in the context of spectral projectors, it is less effective, as it fails to detect concentration along a single tube in the sense of \eqref{BS-microlocal-infor}.
By contrast, the microlocal profile in \eqref{BS-microlocal-infor} itself is essentially $L^\infty$-type and captures such concentration, but it interacts poorly with induction on scales.
Although Blair and Sogge did not use an induction-on-scales argument in their works, the microlocal Kakeya--Nikodym norm $\|\cdot\|_{\MKN(\lambda)}$ fits well within this scheme: it retains the scale-stable advantages of \eqref{eq:L2avg} while still detecting the concentration measured by \eqref{BS-microlocal-infor}. 
The isotropic microlocal Kakeya--Nikodym norm thus appears to provide a natural framework for studying spectral projectors for general manifolds.

The anisotropic refinement of the microlocal Kakeya--Nikodym norm introduced in this paper arises naturally in the context of Euclidean geometry.
In particular, it is essential for the constant-sectional-curvature parts of Theorems \ref{theo main4} and \ref{theo main5}, and it serves to distinguish Euclidean geometry from Riemannian geometry.
Isotropic microlocal Kakeya--Nikodym norms miss the relevant multiscale geometry. The anisotropic variant aligns with the Kakeya-compression picture: obstructions arise when tubes cluster near lower-dimensional sets, and the norm is designed to register concentration along subspaces of any intermediate dimension. 
However, an anisotropic microlocal Kakeya--Nikodym norm is not naturally defined for general phases. For the spectral projectors on a Riemannian manifold, for example, anisotropic microlocal Kakeya--Nikodym averages would need to be taken over neighborhoods of many totally geodesic submanifolds; the existence of such families already requires strong geometric conditions.
In contrast, as observed in \cite{gao2025curved}, on a manifold of constant sectional curvature there are local diffeomorphisms $\Psi\colon U\to\mathbb R^n$ that straighten geodesics to lines and, more generally, map totally geodesic submanifolds to affine subspaces. Consequently, on each such chart we may define the anisotropic microlocal Kakeya--Nikodym norm exactly as in the Euclidean case, with no ambiguity. More importantly, Euclidean local incidence bounds transfer to this setting, yielding the favorable estimates needed to prove \eqref{eq:mica}.

Further refinement on the current anisotropic microlocal Kakeya--Nikodym norm seems possible. 
See Remark \ref{algebraic-rmk}.
We hope the approach developed here sheds light on future investigations.

\begin{remark}
Riemannian manifolds with constant sectional curvature are closely related to phase functions satisfying \emph{Bourgain's condition}, introduced in \cite{guo2024dichotomy} for H\"ormander operators. 
In the special case where the phase comes from the Riemannian distance function, Dai, Gong, Guo, and Zhang \cite{DaiGongGuoZhang:CJM:2024} proved that Bourgain's condition holds if and only if the manifold has constant sectional curvature.
\end{remark}

Finally, we make two remarks on the higher-dimensional analogue of Corollaries \ref{cor extension} and \ref{cor bochner} and possible improvements.

\begin{remark}
\label{p-q-rmk}
\rm
It is possible to prove the optimal $(q,p)$ estimate for the Fourier extension operator for exponents $p$ matching the higher-dimensional state-of-the-art $(p,p)$ results in \cite{Wang-Wu}, by using the Katz--Tao incidence estimate (see \cite[Theorem 3.5]{Wang-Wu}) and by considering a larger mixed norm
\begin{equation}
    \|f\|_{L^2(B_\lambda^{n-1})}^{\,2-\tfrac{2(n+1)}{p(n-1)}}\,
    \sup_{\theta}\|f_\theta\|_{L^2_{\rm avg}}^{\,\tfrac{2(n+1)}{p(n-1)}-1},
\end{equation}
on the right-hand side of \eqref{eq:MKNE}, where $\theta$ ranges over frequency caps of radius $\lambda^{-1/2}$ and $\|{}\cdot{}\|_{L^2_{\rm avg}}$ is as in \eqref{eq:L2avg}. 
\end{remark}

\begin{remark}
\rm
By the reduction in \cite{GOWWZ}, it is possible to obtain the optimal $L^p$ estimate for the Bochner--Riesz operator for exponents $p$ matching the higher-dimensional state-of-the-art $L^p$ estimate for the Fourier extension operator established in \cite{Wang-Wu}.
\end{remark}

\smallskip

\subsection*{Structure of the paper}
Section \ref{sec wave} constructs the wave-packet decompositions and records several auxiliary results. Section \ref{sec general} proves the microlocal Kakeya--Nikodym bounds via induction on scales and incidence estimates. Section \ref{sec constantcurv} establishes the corresponding results for the Fourier extension operator and for spectral projectors on manifolds of constant sectional curvature. Section \ref{sec example} provides the sharpness examples and completes the proof of the necessity statements in Theorem \ref{theo main1}.

\subsection*{Notation}
$B_r^m(x)$ is the $m$-dimensional Euclidean ball of radius $r$ centered at $x$. When the center is irrelevant (and the dimension is clear), we abbreviate $B_r^m$ or simply $B_r$. Let $w_{B_r}$ be a Schwartz weight adapted to $B_r$: $w_{B_r}\ge 1$ on $B_r$, and it decays rapidly off $B_r$.

For a set $E\in\mathbb R^m$ and $r>0$, we use $N_r(E)$ to denote the $r$-neighborhood of $E$ in $\mathbb R^m$.

For nonnegative quantities $A,B$ and a parameter $a$, write $A\lesssim B$ if there exists an absolute constant $C>0$ such that $A\le CB$; write $A\sim B$ if $A\lesssim B$ and $B\lesssim A$; and write $A\lesssim_a B$ if there exists a constant $C_a>0$ such that $A\le C_a B$.

For nonnegative functions $A(R),B(R)$ with $R\ge 1$, write $A(R)\lessapprox_\varepsilon B(R)$ if for any $\varepsilon>0$, there exists $C_\varepsilon>0$  such that $A(R)\le C_\varepsilon R^\varepsilon B(R)$ for all $R\ge 1$.

We use ${\rm RapDec}(R)$ for any quantity that decays faster than any power of $R$: for every $N\in\N$ there exists $C_N$ with $|{\rm RapDec}(R)|\le C_N R^{-N}$.

\subsection*{Acknowledgment} This project was partly supported by the National Key R\&D Program of China under Grant No. 2022YFA1007200 and 2024\hspace{0pt}YFA\hspace{0pt}1015400. C. Gao was supported by the Natural Science Foundation of China under Grant No. 12301121.
S. Wu was partially supported by NSF2453583.
Y. Xi was partially supported by the Natural Science Foundation of China under Grant No. 12571107 and 12171424 and the Zhejiang Provincial Natural Science Foundation of China under Grant No. LR25A010001. 
The authors are grateful to Jonathan Hickman for insightful comments and suggestions. They also thank Betsy Stovall for historical information on sharp $(q,p)$ estimates and Diankun Liu for carefully reading a draft of the paper.

\bigskip

\section{wave-packet decompositions and preliminaries}\label{sec wave}
In this section, we perform the wave-packet decomposition separately for $n\times(n-1)$ H\"ormander oscillatory integral operators satisfying the positive-definite Carleson--Sj\"olin conditions on $\mathbb{R}^n\times\mathbb{R}^{n-1}$ and for the approximate spectral projection operator on $\mathbb{R}^n\times\mathbb{R}^n$. The former follows the Guth--Hickman--Iliopoulou approach \cite{GHI}, in which one partitions frequency space into caps of diameter $\lambda^{-1/2+\e_0}$ and constructs wave-packets localized in corresponding curved tubes. The latter follows Blair--Sogge's wave-packet decomposition \cite{blair2015refined}, employing geodesic normal coordinates and a spherical cap decomposition to build wave-packets supported in $\lambda^{-1/2+\e_0}$-tubes around geodesics, reflecting the microlocal structure of the phase $d_g(x,y)$. 
In the study of spectral projectors, a standard device is to freeze a suitable component of $y$, thereby reducing the $n\times n$ phase $d_g(x,y)$ to an $n\times(n-1)$ H\"ormander-type phase (see, e.g., \cite{burq2007restrictions,gao2024refined}). However, this reduction is incompatible with the Kakeya--Nikodym norm. We therefore perform the wave-packet decomposition in the original $n\times n$ framework.

Once both decompositions are in place, we adopt a unified notation for the wave-packets in contexts where no ambiguity arises. The refined decoupling estimates, which depend only on the microlocal profiles of these packets (see \cite{beltran2020variable, iosevich2022microlocal, gan2025local}), thus apply identically to both settings. Consequently, the incidence-theoretic arguments in later sections treat the $n\times(n-1)$ and $n\times n$ phases on the same footing, and Theorems \ref{theo main2}, \ref{theo main3}, and \ref{theo main5} follow from a single unified proof.

\medskip 

\subsection{H\"ormander operators}
Let $n\ge2$ and let $a\in C_c^\infty(\mathbb{R}^n\times\mathbb{R}^{n-1})$ be nonnegative with $\supp a\subset B_1^n(0)\times B_1^{n-1}(0)$.  Let $\phi\colon B_1^n(0)\times B_1^{n-1}(0)\to\mathbb{R}$ satisfy the positive-definite Carleson--Sj\"olin conditions:
\begin{itemize}
  \item[\H] $\rank\partial^2_{\xi x}\phi(x,\xi)=n-1$ for all $(x,\xi)\in B_1^n(0)\times B_1^{n-1}(0)$;
  \item[\HHH] Writing 
  \[
    G_0(x,\xi)=\bigwedge_{j=1}^{n-1}\partial_{\xi_j}\partial_x\phi(x,\xi),
    \quad
    G(x,\xi)=\frac{G_0(x,\xi)}{|G_0(x,\xi)|},
  \]
  the Hessian 
  \[
    \partial^2_{\xi\xi}\langle\partial_x\phi(x,\xi),G(x,\xi_0)\rangle
    \bigl|_{\xi=\xi_0}
  \]
  is positive-definite for each $(x,\xi_0)\in\supp a$.
\end{itemize}
For $\lambda\ge1$ define
\[
\mathcal H^\lambda f(x)
=\int_{B_1^{n-1}(0)}e^{2\pi i\,\phi^\lambda(x,\xi)}\,a^\lambda(x,\xi)\,f(\xi)\,d\xi,
\]
where $a^\lambda(x,\xi)=a(x/\lambda,\xi)$ and $\phi^\lambda(x,\xi)=\lambda\phi(x/\lambda,\xi)$.
As is standard in induction-on-scales arguments for variable-coefficient problems, we introduce an intermediate scale $1\le R\le \lambda$ to run the multiscale scheme. 

Throughout this section, we fix a small number 
\begin{equation}
\label{eps_0}
    \e_0>0
\end{equation}
to be specified (later, we will choose $\e_0=\e^{1000}$, where $\e>0$ is the given number in the statement of Propositions \ref{prop main1} and \ref{prop main2}).
Let $s$ be a dyadic number with
\[
R^{-\tfrac12+\varepsilon_0}\le s\le 1.
\]
Let $0\le \beta\in C_0^\infty(\mathbb R^{n-1})$ satisfy
\[
\sum_{\nu\in\mathbb Z^{n-1}} \beta(z+\nu)=1,\quad \operatorname{supp}\beta\subset \{x\in\mathbb R^{n-1}: |x|\le 2\}.
\]
Decompose frequency space $\mathbb R^{n-1}$ into cubes $\{\theta\}$ by choosing centers $\xi_\theta\in s\mathbb Z^{n-1}$ and defining
\[
f_\theta(\xi)=\beta(s^{-1}(\xi-\xi_\theta)) f(\xi).
\]
For each $v\in s \mathbb Z^{n-1}$ set
\[
Q^{R,s}_{\theta,v}f(\xi)
=\bigl[(f_\theta)^\vee({}\cdot{})\,\beta(s^{-1}R^{-1}({}\cdot{}-Rv))\bigr]^\wedge(\xi).
\]
Then 
\[
f=\sum_\theta\sum_v Q^{R,s}_{\theta,v}f,
\]
and since the frequency and spatial cutoffs have bounded overlap, Plancherel gives
\begin{equation}
\nonumber
\|f\|_{L^2}^2\sim \sum_{\theta,v}\|Q^{R,s}_{\theta,v}f\|_{L^2}^2,
\end{equation}
with implicit constants depending only on the dimension and on $\beta$.

Let $ \xi_\theta$ be the center of $ \theta$. By the Carleson--Sj\"olin condition, we write $ x=(x',x_n)$ with $ x'\in \mathbb R^{n-1}$ and $ x_n\in \mathbb R$, such that there is a neighborhood of $0$ in $ \mathbb R^{n-1}$,where for each $ w$ with $ |w|\le 1$ and each $ |x_n|\le 1$ one can solve uniquely for a smooth function $ \gamma_\theta(w,x_n)\in \mathbb R^{n-1}$ satisfying
\[
\partial_\xi\phi\bigl((\gamma_\theta(w,x_n),x_n),\xi_\theta\bigr)= w.
\]
Fix a ball $B_R\subset B_\lambda$ with $1\le R\le \lambda$. Without loss of generality, we may assume $B_R$ is centered at the origin. For $v$ with $|v|\le 1$ define the rescaled curve
\[
\gamma^\lambda_{\theta}(v,x_n):=\lambda\,\gamma_\theta(Rv,x_n/\lambda),
\]
which lies in $B_R\subset B_\lambda$. Consider the tubular neighborhood of $\gamma^\lambda_{\theta}$ in $B_R$,
\[
T^{R,s}_{\theta,v}
=\{(x',x_n)\in B_R:\ |x'-\gamma^\lambda_{\theta}(v,x_n)|\le C_0\,sR,\ |x_n|\le R\}
,
\]
where $C_0>0$ is a large absolute constant that only depends on $\phi$, and will be chosen later.

\begin{lemma}[Wave-packet support]\label{lem:wp_support}
For $x\in B_R\setminus T^{R,s}_{\theta,v}$,
\[
|\mathcal H^\lambda (Q^{R,s}_{\theta,v}f)(x)|={\rm RapDec}(R)\|f\|_{L^2},
\]
uniformly for $R^{-\tfrac12+\varepsilon_0}\le s\le 1$.

\end{lemma}

\begin{proof}
We write
\begin{align*}
\mathcal H^\lambda (Q^{R,s}_{\theta,v}f)(x)
&=
\iiint 
e^{2\pi i(\phi^\lambda(x,\xi)-y\cdot(\xi-\eta))}
\,a^\lambda(x,\xi)\,
\beta\bigl(s^{-1}R^{-1}(y-Rv)\bigr)\\
&\quad\times \beta\bigl(s^{-1}(\eta-\xi_\theta)\bigr)
f(\eta)\,d\eta\,dy\,d\xi.
\end{align*}

By the support of the second bump,
\[
|\eta-\xi_\theta|\lesssim s.
\]
Moreover, since $y$ is localized to a ball of radius $\sim sR$, the Fourier transform in $y$ shows that the integral is negligible unless
\[
|\xi-\eta|\lesssim s^{-1}R^{-1+\varepsilon_0}\le R^{-1/2}\le s.
\]
Consequently
\[
|\xi-\xi_\theta|\lesssim  s
\]
as well. 

If $x\notin T^{R,s}_{\theta,v}$ then $|x'-\gamma^\lambda_{\theta}(v,x_n)|> sR$. By the mean value theorem in the $x'$ variable there is $\tilde x'$ on the segment between $x'$ and $\gamma^\lambda_{\theta}(v,x_n)$ such that
\[
\partial_\xi\phi^\lambda(x,\xi)
-\partial_\xi\phi^\lambda((\gamma^\lambda_{\theta}(v,x_n),x_n),\xi)
= \partial_{x'}\partial_\xi\phi^\lambda((\tilde x',x_n),\xi)\cdot (x'-\gamma^\lambda_{\theta}(v,x_n)).
\]
Using the uniform lower bound on the singular values of $\partial_{x'}\partial_\xi\phi^\lambda$ in the relevant region,
\[
\bigl|\partial_\xi\phi^\lambda(x,\xi)
-\partial_\xi\phi^\lambda((\gamma^\lambda_{\theta}(v,x_n),x_n),\xi)\bigr|
\ge c\,|x'-\gamma^\lambda_{\theta}(v,x_n)|.
\]
Since $\partial_\xi\phi^\lambda((\gamma^\lambda_{\theta}(v,x_n),x_n),\xi_\theta)=Rv$, it follows that
\[
|\partial_\xi\phi^\lambda(x,\xi)-Rv|
\gtrsim sR.
\]
Because $|y-Rv|\lesssim sR$, if we chose $C_0>0$ to be large enough, we then have
\[
|\partial_\xi\phi^\lambda(x,\xi)-y|
\ge |\partial_\xi\phi^\lambda(x,\xi)-Rv| - |y-Rv|
\gtrsim sR.
\]
Thus the phase in $\xi$ is nonstationary with parameter $\gtrsim sR$. Repeated integration by parts in $\xi$, using the operator
\[
L=\frac{1}{2\pi i}\frac{(\partial_\xi\phi^\lambda(x,\xi)-y)\cdot \nabla_\xi}{|\partial_\xi\phi^\lambda(x,\xi)-y|^2},
\]
(which satisfies $L e^{2\pi i(\phi^\lambda(x,\xi)-y\cdot\xi)}=e^{2\pi i(\phi^\lambda(x,\xi)-y\cdot\xi)}$) moves derivatives onto the smooth amplitude and cutoff factors; each application gains a factor $(s^2R)^{-1}\le R^{-2\varepsilon_0}$, iterating gives arbitrary decay:
\[
|\mathcal H^\lambda (Q^{R,s}_{\theta,v}f)(x)|={\rm RapDec}(R)\|f\|_{L^2}.
\]
\end{proof}

\begin{definition}\label{def:MKN_s}
Let $s$ range over dyadic values with
\[
R^{-\tfrac12+\varepsilon_0}\le s\le 1.
\]
For $f\in L^2(\mathbb R^{n-1})$, define
\[
\|f\|_{\MKN_{\text{\rm H\"orm}}(R)}
:=\sup_{s,\theta,v}s^{-\frac{n-1}{2}}
\,\|Q^{R,s}_{\theta,v}f\|_{L^2(\mathbb R^{n-1})},
\]
where the supremum in $\theta,v$ is over frequency cubes $\theta$ with centers $\xi_\theta\in s\mathbb Z^{n-1}$ and spatial shifts $v\in s\mathbb Z^{n-1}$. 
\end{definition}

\smallskip

In the special case $\mathcal H^\lambda=E_\Sigma$ (the Fourier extension operator), we define an anisotropic version as follows. Let
\[
\vec s=(\mathbf s;U),
\]
where $\mathbf s=(s_1,\dots,s_{n-1})$ with each $s_i$ dyadic and $R^{-1/2+\varepsilon_0}\le s_i\le 1$, and $U\in O(n-1)$, the orthogonal group on $\mathbb R^{n-1}$. Choose $\beta_0\in C_0^\infty(\mathbb R)$ with
\[
\sum_{\nu\in\mathbb Z}\beta_0(t+\nu)=1,\qquad \operatorname{supp}\beta_0\subset\{t:\ |t|\le 2\}.
\]
Define the oriented bump
\[
\beta_{\vec s}(z):=\beta_{\mathbf s}(U^{-1}z),
\qquad\text{where }\ 
\beta_{\mathbf s}(z):=\prod_{i=1}^{n-1}\beta_0(s_i^{-1}z_i),
\quad z\in\mathbb R^{n-1}.
\]
Set the rotated lattice
\[
\Lambda_{\vec s}:=\bigl\{\,U\,\mathrm{diag}(s_1,\dots,s_{n-1})\,m:\ m\in\mathbb Z^{n-1}\,\bigr\}.
\]
Choose frequency rectangles $\theta$ with centers $\xi_\theta\in \Lambda_{\vec s}$ and define
\[
f_{\theta,\vec s}(\xi)=\beta_{\vec s}(\xi-\xi_\theta)\,f(\xi).
\]

For spatial shifts  $v\in \Lambda_{\vec s}$ define
\[
Q^{R,\vec s}_{\theta,v}f(\xi)
=\Bigl[(f_{\theta,\vec s})^\vee({}\cdot{})\,\beta_{\vec s}\bigl(R^{-1}({}\cdot{}-Rv)\bigr)\Bigr]^\wedge(\xi).
\]
Then
\[
f=\sum_{\theta,v}Q^{R,\vec s}_{\theta,v}f,
\qquad
\|f\|_{L^2}^2\sim \sum_{\theta,v}\|Q^{R,\vec s}_{\theta,v}f\|_{L^2}^2,
\]
with implicit constants depending only on $n$ and $\beta_0$.

Given a ball $B_R\subset B_\lambda$. At scale $\vec s$, each wave-packet $E_\Sigma\bigl(Q^{R,\vec s}_{\theta,v}f\bigr)$ is microlocalized to a rectangular box in physical space $B_R$ of dimensions
\[
R \times s_1R \times s_2R \times \cdots \times s_{n-1}R,
\]
whose long axis is the propagation direction and whose transverse axes have lengths determined by $\vec s$ and directions determined by $U$.
We define the anisotropic microlocal Kakeya--Nikodym norm by
\[
\|f\|_{\widetilde{\MKN}(R)}
:=\sup_{\vec s}\,\sup_{\theta,v}
\ \Bigl(\prod_{i=1}^{n-1}s_i^{-1/2}\Bigr)\,\|Q^{R,\vec s}_{\theta,v}f\|_{L^2(B^{n-1}_1)}.
\]
For fixed $\vec s$, the wave-packet support lemma and Plancherel hold for the families $E_\Sigma(Q^{R,\vec s}_{\theta,v}f)$ and $Q^{R,\vec s}_{\theta,v}f$, respectively. In fact, each $E_\Sigma(Q^{R,\vec s}_{\theta,v}f)$ is essentially supported in a rectangular box $\Box^{R,\vec s}_{\theta,v}$ adapted to it.

\medskip

\subsection{Spectral projectors}\label{subsec:spec-proj} Now we perform decompositions for the approximate spectral projector.
Recall that
$$e_\lambda=\chi_\lambda e_\lambda=\chi(\lambda-\sqrt{\Delta_g})\,e_\lambda,$$ for $\chi\in\mathcal S$
satisfying $\chi(0)=1$ and $\operatorname{supp}\widehat\chi\subset(\frac12,1)$. Then by \cite[Lemma 5.1.3]{sogge2017fourier}, for any $N\in\mathbb N$ and $\lambda\ge1$ one has
\begin{multline}
\chi_\lambda f(x)
=\frac{1}{2\pi}\int \widehat\chi(t)\,e^{i\lambda t}\bigl(e^{-it\sqrt{\Delta_g}}f\bigr)(x)\,dt
\\=\lambda^{\frac{n-1}{2}}\int e^{i\lambda\,d_g(x,y)}\,a_\lambda(x,y)\,f(y)\,dy
+{\rm RapDec}(\lambda)\|f\|_{L^2},
\end{multline}
where $a_\lambda\in C^\infty$ ranges in a bounded set and $a_\lambda(x,y)=0$ if $d_g(x,y)\notin(\frac12,1)$. It therefore suffices to work with the operator
\[
\widetilde\chi_\lambda  f(x):=
\int e^{i\lambda\,d_g(x,y)}\,a_\lambda(x,y)\,f(y)\,dy,
\]
in a single coordinate chart. The phase $d_g(x,y)$ satisfies the positive-definite $n\times n$ Carleson--Sj\"olin hypotheses (cf.\ \cite{sogge2017fourier}). Consequently, the Blair--Sogge wave-packet construction applies to any phase in this class. For simplicity, we carry out the decomposition only for the distance phase $d_g(x,y)$; the extension to a general positive-definite $n\times n$ Carleson--Sj\"olin phase follows with routine modifications.
Since the problem is purely local, after a partition of unity and, if needed, a rescaling of the metric, we work in geodesic normal coordinates about a fixed point and assume that the coefficients $g_{ij}$ are uniformly close to the Euclidean ones on this chart. By a second partition of unity in $(x,y)$, we further assume that $a_\lambda(x,y)=0$ unless $x\in B_{\frac1{10}}(0)$ and $y\in B_{\frac1{10}}\bigl((0,\dots,0,\frac12)\bigr)$.

Let $\Phi_t(x,\xi)=(y(t),\zeta(t))$ be the geodesic flow on $S^*M$ with $(y(0),\zeta(0))=(x,\xi)$. For a unit cotangent vector $(x,\xi)$ with $x\in B_1(0)$ in the chart and $\xi=(\xi',\xi_n)$ with $|\xi_n|\ge \tfrac12$, there exists a unique time $t=t(x,\xi)\in(-1,1)$ such that $y_n(t)=0$. Denote $$v(x,\xi):=(y_1(t(x,\xi)),\dots,y_{n-1}(t(x,\xi)))\in\mathbb R^{n-1}.$$ Set $$\omega(x,\xi)=\zeta(t(x,\xi))/|\zeta(t(x,\xi))|\in S^{n-1},$$ and write $\omega(x,\xi)=(\omega'(x,\xi),\omega_n(x,\xi))$. Then $\omega_n(x,\xi)$ is bounded away from $0$ on the region of interest.

Let $1\le R\le \lambda$. Fix $0<\varepsilon_0\ll1$. Let $\beta\in C_0^\infty(\mathbb R^{n-1})$ be as in the previous section, with 
\[
\sum_{\nu\in\mathbb Z^{n-1}}\beta(z+\nu)=1,\qquad \operatorname{supp}\beta\subset\{z:\ |z|\le 2\}.
\]
Let $\alpha\in C_0^\infty((c/2,2c^{-1}))$ satisfy $\alpha\equiv1$ on $[c,c^{-1}]$ for a small fixed $c>0$. For dyadic $s$ with $R^{-1/2+\varepsilon_0}\le s\le 1$, let $v\in s\mathbb Z^{n-1}$. We decompose $\mathbb R^{n-1}$ into cubes $\{\theta\}$ with centers $\xi_\theta\in s\mathbb Z^{n-1}$, define the microlocal cutoff by
\[
\mathcal Q^{R,s}_{\theta,v}(y,\eta)
=
\beta\bigl(s^{-1}(\omega'(y,\eta)-\theta)\bigr)\,
\beta\bigl(s^{-1}(v(y,\eta)-v)\bigr)\,
\alpha(|\eta|/\lambda),
\]
and let $\mathcal Q^{R,s}_{\theta,v}(y,D)$ be the associated order-zero pseudodifferential operator of type $(\tfrac12+\varepsilon_0,\tfrac12-\varepsilon_0)$. 
For $f\in L^2 \bigl(B\bigl((0,\dots,0,\tfrac12),\tfrac1{10}\bigr)\bigr)$ with Fourier support in the set $\{\eta:\ |\eta|\in[c,c^{-1}]\}$, and any $N\in\mathbb N$, we have
\[
f(y)=\sum_{\theta,v} \mathcal Q^{R,s}_{\theta,v}(y,D)f(y)+{\rm RapDec}(\lambda)\|f\|_{L^2},
\ 
\|f\|_{L^2}^2\sim \sum_{\theta,v}\|\mathcal Q^{R,s}_{\theta,v}(y,D)f(y)\|_{L^2}^2,
\]
and each $\mathcal Q^{R,s}_{\theta,v}(y,D)f(y)$ is microlocally supported near a single geodesic tube of length $1$ and cross-section radius $s$.

Define the geodesic $\gamma_{\theta,v}$ to be the unit-speed geodesic with
\[
\gamma_{\theta,v}(0)=(v,0)\in\{x_n=0\},
\qquad 
\dot\gamma_{\theta,v}(0)=\omega_\theta,
\]
where $\omega_\theta=(\omega'_\theta,\omega_{\theta,n})$ is a unit vector with $\omega'_\theta=\theta$. Consider on $B_\lambda$ the rescaled operator
\[
\chi^\lambda f(x):=\widetilde\chi_\lambda f(x/\lambda).
\]
Fix a ball $B_R\subset B_\lambda$. After the dilation $x\mapsto \lambda x$, write again $x$ for the rescaled variable and set
\[
\gamma^{\lambda}_{\theta,v}=\{\lambda\,\gamma_{\theta,v}(t/\lambda):\ |t|\le \lambda\}\subset B_\lambda,\ 
T^{R,s}_{\theta,v}=\{x\in B_R:\ \dist(x,\gamma^{\lambda}_{\theta,v})\le C_0\,sR\},
\]
where $C_0>0$ is a large absolute constant.
 Therefore, $T^{R,s}_{\theta,v}$ is a curved tube of length $\sim R$ and radius $\sim sR$ about the rescaled geodesic $\gamma^{\lambda}_{\theta,v}$ in $B_R$.

These wave-packets obey the same support and almost-orthogonality properties.

\begin{lemma}[Wave-packet support]\label{lem:wp_support_spec}
For $R^{-1/2+\varepsilon_0}\le s\le 1$,
\[
\bigl|\chi^\lambda\bigl(\mathcal Q^{R,s}_{\theta,v}f\bigr)(x)\bigr|
={\rm RapDec}(\lambda)\,\|f\|_{L^2}
\quad\text{whenever } x\in B_R\setminus T^{R,s}_{\theta,v}.
\]
\end{lemma}

The proof of Lemma \ref{lem:wp_support_spec} is essentially the same as that of Lemma \ref{lem:wp_support}; see also Lemma 3.2 in \cite{blair2015refined} for a direct and more detailed argument.

\begin{lemma}\label{lem:L2_orth_spec}
For each fixed $s$ with $R^{-1/2+\varepsilon_0}\le s\le 1$,
\begin{equation}
\nonumber
    \Bigl\|\sum_{\theta,v}\mathcal Q^{R,s}_{\theta,v}f\Bigr\|_{L^2}^2
\sim \sum_{\theta,v}\|\mathcal Q^{R,s}_{\theta,v}f\|_{L^2}^2,
\end{equation}
with implicit constants independent of $\lambda$, $R$, and $s$.
\end{lemma}

\begin{proof}
This follows from the $L^2$ boundedness of the order-zero operators $\mathcal Q^{R,s}_{\theta,v}(x,D)$, the type $(\tfrac12+\varepsilon_0,\tfrac12-\varepsilon_0)$ calculus, and the bounded overlap of the supports of the symbols $\mathcal Q^{R,s}_{\theta,v}(x,\xi)$ in $(x,\xi)$.
\end{proof}

Similarly, we define the associated microlocal Kakeya--Nikodym  norm as follows.

\begin{definition}[Microlocal Kakeya--Nikodym  norm for the projector]\label{def:MKN_spec}
For $f\in L^2$ supported in the chart, define
\[
\|f\|_{\MKN_{\mathrm{spec}}(R)}
:=\sup_{R^{-1/2+\varepsilon_0}\le s\le 1}\ \sup_{\theta,v}
s^{-\frac{n-1}{2}}\,\|\mathcal Q^{R,s}_{\theta,v}(x,D)f\|_{L^2}.
\]
\end{definition}

Now we define the anisotropic microlocal Kakeya--Nikodym norm for the constant sectional curvature case.
We choose a chart where all geodesics are straight lines and $\{|\eta_n|\ge \tfrac12|\eta|\}$. We parametrize direction and transverse position by
\[
u(\eta)=\frac{\eta'}{\eta_n}\in\mathbb R^{n-1},\qquad
b(y,\eta)=y'-y_n\,\frac{\eta'}{\eta_n}\in\mathbb R^{n-1}.
\]
Fix $0<\varepsilon_0\ll1$. Recall that $\beta_0\in C_0^\infty(\mathbb R)$ satisfies
\[
\sum_{\nu\in\mathbb Z}\beta_0(t+\nu)=1,\qquad \supp\beta_0\subset\{t:|t|\le 2\},
\]
and, for $\vec s=(\mathbf s;U)(s_1,\dots,s_{n-1};U)$ with each $s_i$ dyadic and $R^{-1/2+\varepsilon_0}\le s_i\le 1$, $U\in O(n-1)$, set
\[
\beta_{\mathbf s}(z)=\prod_{i=1}^{n-1}\beta_0(s_i^{-1}z_i),\qquad z\in\mathbb R^{n-1}.
\]
Define
\[
\beta_{\vec s}(z):=\beta_{\mathbf s}(U^{-1}z),\qquad z\in\mathbb R^{n-1}.
\]
Choose rectangular grids
\[
\theta\in U(\vec s\,\mathbb Z^{n-1}) ,\qquad
v\in U(\vec s\,\mathbb Z^{n-1}).
\]
Define the microlocal cutoffs as:
\[
\mathcal Q^{R,\vec s}_{\theta,v}(y,\eta)
=\beta_{\vec s}\bigl(u(\eta)-\theta\bigr)\,
 \beta_{\vec s}\bigl(b(y,\eta)-v\bigr)\,
 \alpha(|\eta|/\lambda),
\]
and let $\mathcal Q^{R,\vec s}_{\theta,v}(y,D)$ be the associated zero-order pseudodifferential operator of type $(\tfrac12+\varepsilon_0,\tfrac12-\varepsilon_0)$. Then, for $f$ supported in the chart with Fourier support in $\{\eta:|\eta|\in[c,c^{-1}]\}$,
\[
f=\sum_{\theta,v}\mathcal Q^{R,\vec s}_{\theta,v}(y,D)f+{\rm RapDec}(\lambda)\|f\|_{L^2},
\qquad
\|f\|_{L^2}^2\sim \sum_{\theta,v}\|\mathcal Q^{R,\vec s}_{\theta,v}(y,D)f\|_{L^2}^2.
\]
Each packet $\chi^\lambda\bigl(\mathcal Q^{R,\vec s}_{\theta,v}f\bigr)$ restricted to a ball $B_R$ is microlocalized in a rotated rectangular box of dimensions
\[
R \times s_1R \times s_2R \times \cdots \times s_{n-1}R.
\]

Finally, we set
\[
\|f\|_{\widetilde{\MKN}_{\mathrm{spec}}(R)}
:=\sup_{\vec s}\, \sup_{\theta,v}
\ \Bigl(\prod_{i=1}^{n-1}s_i^{-1/2}\Bigr)\,\|\mathcal Q^{R,\vec s}_{\theta,v}(x,D)f\|_{L^2}.
\]
For fixed $\vec s=(\mathbf s;U)$, the wave-packet support lemma and Plancherel hold for the families $\chi^\lambda(\mathcal Q^{R,\vec s}_{\theta,v}f)$ and $\mathcal Q^{R,\vec s}_{\theta,v}f$, respectively, and each $\chi^\lambda(\mathcal Q^{R,\vec s}_{\theta,v}f)$ is essentially supported in the rectangular box $\Box^{R,\vec s}_{\theta,v}$ adapted to it.

\medskip

\subsection{Unified notation}
From now on we use a single notation that covers both settings (the H\"ormander $n\times(n-1)$ operators and the spectral projector with phase $d_g$). This is possible because the wave-packets in the two cases have the same microlocal profile, hence satisfy the same refined decoupling estimates. The notation is summarized in Table \ref{tab:notation-general}.

For the sake of simplicity, we often use a unified, simplified notation as in the table. When $R$, $s$, and the tube $T=T^s:=T^{R,s}_{\theta,v}$ are clear from context, we write $f_T:=f^{R,s}_{\theta,v}$. We also use $T$ and $(\theta,v)$ interchangeably to index wave-packets and abbreviate $\sum_T$ for $\sum_{\theta,v}$.

Now, in this notation, we summarize our previous findings from the wave-packet decomposition in the following lemmas.

\begin{lemma}[Orthogonality]\label{lem Orthogonal}
For a fixed $s\in[R^{-1/2+\varepsilon_0},1]$,
\begin{equation}\label{eq:ortho1}
\bigl\|\sum_{\theta,v}f_{\theta,v}^{R,s} \bigr\|_{L^2}^2\sim \sum_{\theta,v}\|f_{\theta,v}^{R,s}\|_{L^2}^2.
\end{equation}

\end{lemma}

\begin{lemma}[Wave-packet support]\label{lem:WSI}
\label{wpt-spt-lem} For a fixed $s\in[R^{-1/2+\varepsilon_0},1]$,
$\mathcal T^\lambda f^{R,s}_{\theta,v}$ is concentrated in the tube $T^{R,s}_{\theta,v}$, in the sense that,
\[
\bigl|\mathcal T^\lambda f^{R,s}_{\theta,v}(x)\bigr|={\rm RapDec}(R)\|f\|_{L^2}
\quad\text{whenever } x\in B_R\setminus T^{R,s}_{\theta,v}.
\]
\end{lemma}

In unified notation, the microlocal Kakeya--Nikodym norm is given by
\[
\|f\|_{\MKN(R)}
:=\sup_{R^{-1/2+\varepsilon_0}\le s\le 1}\ \sup_{T^s}
s^{-\frac{n-1}{2}}\,\|f_{T^s}\|_{L^2}.
\]
This norm is closely related to the notation ``$s$-dimensional ball" condition in incidence geometry.
See, for example, \cite[Remark 4.8]{gan2025local}. 
In particular, we have the following lemma.
\begin{lemma}\label{KN-L2-lem}
Let $f=\sum_{T\in\mathbb T} f_T$ be a sum of wave-packets at scale $s=R^{-1/2+\varepsilon_0}$. Then
\begin{equation}\label{eq:MKNbound}
    \|f\|_{L^2}^2 \lesssim R^{O(\varepsilon_0)}\,
    \frac{\#\mathbb T}{R^{\frac{n-1}{2}}}\,
    \|f\|_{\MKN(R)}^2.
\end{equation}
\end{lemma}

\begin{proof}
By $L^2$ almost-orthogonality \eqref{eq:ortho1} at scale $s=R^{-1/2+\varepsilon_0}$,
\[
    \|f\|_{L^2}^2 \sim \sum_{T\in \mathbb{T}}\|f_{T}\|_{L^2}^2
    \lesssim R^{O(\varepsilon_0)}\,\#\mathbb T\ \sup_{T\in\mathbb T}\|f_T\|_{L^2}^2.
\]
By the definition of the microlocal Kakeya--Nikodym norm at scale $s$,
\[
\|f_T\|_{L^2}\lesssim s^{\frac{n-1}{2}}\ \|f\|_{\MKN(R)}\qquad\text{for each }T.
\]
Hence
\[
\sup_{T\in\mathbb T}\|f_T\|_{L^2}^2
\lesssim s^{\,n-1}\,\|f\|_{\MKN(R)}^2
=R^{-\frac{n-1}{2}+O(\varepsilon_0)}\,\|f\|_{\MKN(R)}^2.
\]
Combining the last two inequalities yields \eqref{eq:MKNbound}.
\end{proof}

\begin{table}[t]
\centering
\caption{Notation for the general cases.}
\label{tab:notation-general}
{\setlength{\tabcolsep}{5pt}
 \renewcommand{\arraystretch}{1.05}
 \small
\begin{tabular}{|l|c|c|c|}
\hline
 & H\"ormander operator & Spectral projector & Unified notation \\
\hline
Operator
& $\mathcal H^\lambda$
& $\chi^\lambda$
& $\mathcal T^\lambda$ \\
Curve at scale 1
& $\gamma_{\theta}(v,{}\cdot{})$ from $\phi$
& Geodesic $\gamma_{\theta,v}$
& $\gamma_{\theta,v}$ \\
Curve at scale $\lambda$
& $\gamma^\lambda_{\theta}(v,{}\cdot{})$ from $\phi$
& Geodesic $\gamma^\lambda_{\theta,v}$
& $\gamma^\lambda_{\theta,v}$ \\
Tube
& $T^{R,s}_{\theta,v}$
& $T^{R,s}_{\theta,v}$
& $T:=T^s:=T^{R,s}_{\theta,v}$ \\
Decomposition
& $f=\sum Q^{R,s}_{\theta,v}f$
& $f=\sum \mathcal Q^{R,s}_{\theta,v}f$
& $f=\sum f^{R,s}_{\theta,v}=\sum f_T$ \\
Wave-packets
& $\mathcal H^\lambda Q^{R,s}_{\theta,v}f$
& $\chi^\lambda \mathcal Q^{R,s}_{\theta,v}f$
& $\mathcal T^\lambda f^{R,s}_{\theta,v}=\mathcal T^\lambda f_T$ \\
$\MKN$ norm
& $\|{}\cdot{}\|_{\MKN_{\text{\rm H\"orm}}(R)}$
& $\|{}\cdot{}\|_{\MKN_{\mathrm{spec}}(R)}$
& $\|{}\cdot{}\|_{\MKN(R)}$ \\
\hline
\end{tabular}
}
\end{table}

\smallskip

We also record, in Table \ref{tab:notation-straight}, the unified notation used in the parts of the proof where anisotropic microlocal Kakeya--Nikodym norms are needed, and we state the corresponding lemmas as follows.

\begin{lemma}[Wave-packet support (anisotropic)]\label{lem:WSA}For a fixed $\vec s=(\mathbf s; U)$ where $\mathbf  s\in[R^{-1/2+\varepsilon_0},1]^{n-1}$ and $U\in O(n-1)$,
$\mathcal T^\lambda f^{R,\vec s}_{\theta,v}$ is essentially concentrated in the box $\Box^{R,\vec s}_{\theta,v}$, in the sense that
\[
\bigl|\mathcal T^\lambda f^{R,\vec s}_{\theta,v}(x)\bigr|={\rm RapDec}(R)\|f\|_{L^2}
\quad\text{whenever } x\in B_R\setminus\Box^{R,\vec s}_{\theta,v}.
\]
\end{lemma}
In unified notation, the anisotropic microlocal Kakeya--Nikodym norm is given by
\[
\|f\|_{\widetilde\MKN(R)}
:=\sup_{\vec s}\, \sup_{\Box^{\vec s}}
\ \Bigl(\prod_{i=1}^{n-1}s_i^{-1/2}\Bigr)\,\|f_{\Box^{\vec s}}\|_{L^2}.
\]
This anisotropic norm is closely related to the following definition in incidence geometry.

\begin{table}[t]
\centering
\caption{Additional notation for the constant-curvature cases.}
\label{tab:notation-straight}
{\setlength{\tabcolsep}{5pt}
 \renewcommand{\arraystretch}{1.05}
 \small
\begin{tabular}{|l|c|c|c|}
\hline
 & Euclidean space & Compact space form & Unified notation \\
\hline
Operator
& $E_\Sigma$
& $\chi^\lambda$
& $\mathcal T^\lambda$ \\
Line at scale 1
& Straight line $\gamma_{\theta,v}$
& Straight line $\gamma_{\theta,v}$
& Straight line $\gamma_{\theta,v}$ \\
Line at scale $\lambda$
& Straight line $\gamma^\lambda_{\theta,v}$
& Straight line $\gamma^\lambda_{\theta,v}$
& Straight line $\gamma^\lambda_{\theta,v}$ \\
Box
& $\Box^{R,\vec s}_{\theta,v}$
& $\Box^{R,\vec s}_{\theta,v}$
& $\Box:=\Box^{\vec s}:=\Box^{R,\vec s}_{\theta,v}$ \\
Wave-packets
& $E_\Sigma Q^{R,\vec s}_{\theta,v}f$
& $\chi^\lambda \mathcal Q^{R,\vec s}_{\theta,v}f$
& $\mathcal T^\lambda f_\Box:=\mathcal T^\lambda f^{R,\vec s}_{\theta,v}$ \\
$\widetilde{\MKN}$ norm
& $\|{}\cdot{}\|_{\widetilde{\MKN}(R)}$
& $\|{}\cdot{}\|_{\widetilde{\MKN}_{\mathrm{spec}}(R)}$
& $\|{}\cdot{}\|_{\widetilde{\MKN}(R)}$ \\
\hline
\end{tabular}
}
\end{table}

\begin{definition}
Let $\ZT$ be a family of unit-length $\de$-tubes in $\ZR^n$.
We say $\ZT$ {\bf obeys Wolff's axiom with error $m$} if any $1\times s_1\times\cdots\times s_{n-1}$ rectangular box contains $\lesssim m\cdot s_1\cdots s_{n-1}\,\de^{-(n-1)}$ many $\de$-tubes from $\ZT$.
\end{definition}

In fact, if $f=\sum_{T\in\ZT}f_T$ is a sum of scale $s=R^{-1/2+\e_0}$ wave-packets, then $\|f\|_{\widetilde{\MKN}(R)}$ can be used to control the error $m$, where $m$ is defined so that the $R^{-1}$-dilate of the tubes $T\in\ZT$ satisfies Wolff's axiom (up to the harmless factor $R^{O(\e_0)}$ coming from $\de=R^{-1/2+\e_0}$).
This is shown in the next lemma.
\begin{lemma}\label{KN-L2-lem2}
Let $f=\sum_{T\in\ZT}f_T$ be a sum of wave-packets at scale $s=R^{-1/2+\varepsilon_0}$ such that $\|f_T\|_2$ are comparable for all $T\in\ZT$.
Define
\[
m:=\sup_{\substack{\Box:\ R\times Rs_1\times\cdots\times Rs_{n-1}\text{-box}\\ s_j\in[R^{-1/2+\e_0},\,1]}}
\frac{\#\{T\in\ZT:\ T\subset \Box\}}{R^{-\frac{n+1}{2}}|\Box|}\,.
\]
Then, up to a factor $R^{O(\e_0)}$, the $R^{-1}$-dilate of $\ZT$ obeys Wolff's axiom with error $m$, and
\[
\|f\|_2^2\lesssim R^{O(\e_0)}\,
\frac{\#\mathbb{T}}{m\,R^{\frac{n-1}{2}}}\,
\|f\|_{\widetilde{\MKN}(R)}^2.
\]
\end{lemma}

\begin{proof}
Let $\gamma>0$ be such that $\|f_T\|_2\sim \gamma$ for all $T\in\ZT$.
By $L^2$ orthogonality at the fixed scale $s$,
\[
\|f\|_2^2\sim \sum_{T\in\ZT}\|f_T\|_2^2 \sim (\#\ZT)\,\gamma^2.
\]

By the definition of $m$, there exists a box $\Box$ of the form $R\times Rs_1\times\cdots\times Rs_{n-1}$ for which
\[
\#\{T\in\ZT:\ T\subset \Box\}\gtrsim R^{-\frac{n+1}{2}}|\Box|\, m
\sim R^{\frac{n-1}{2}}\,s_1s_2\cdots s_{n-1}\,m,
\]
up to a factor $R^{O(\e_0)}$, arising from taking $\de=R^{-1/2+\e_0}$ when comparing with the $\de^{-(n-1)}$ term in Wolff's axiom.

By the definition of the anisotropic microlocal Kakeya--Nikodym norm,
\[
\#\{T\in\ZT:\ T\subset \Box\}\cdot\gamma^2\lesssim(s_1s_2\cdots s_{n-1})\,\|f\|_{\widetilde{\MKN}(R)}^2.
\]
Combining with the upper bound on $s_1\cdots s_{n-1}$ yields
\[
\gamma^2\lesssim\frac{1}{m\,R^{\frac{n-1}{2}}}\,\|f\|_{\widetilde{\MKN}(R)}^2,
\]
again up to a factor $R^{O(\e_0)}$.

Finally, multiply both sides by $\#\ZT$ to obtain
\[
(\#\ZT)\,\gamma^2\lesssim R^{O(\e_0)}\,
\frac{\#\ZT}{m\,R^{\frac{n-1}{2}}}\,\|f\|_{\widetilde{\MKN}(R)}^2,
\]
which is the desired estimate because $\|f\|_2^2\sim (\#\ZT)\gamma^2$.
\end{proof}

\smallskip

Now, in unified notation, we state two propositions that imply our main theorems.
\begin{proposition}[Microlocal Kakeya--Nikodym estimates for the general case]\label{prop main1}
For every $\varepsilon>0$ and $\lambda\ge1$, and for all $2\,\frac{3n+1}{3n-3}\le p\le2\,\frac{n+1}{n-1}$,
\begin{equation}\label{eq:prop1}
    \|\mathcal T^\lambda f\|_{L^p(B^n_\lambda)}\lesssim_\varepsilon\ \lambda^\varepsilon\, \|f\|_{L^2(B^{n-1}_1)}^{2-\frac{2(n+1)}{p(n-1)}}\,
    \|f\|_{\MKN(\lambda)}^{\frac{2(n+1)}{p(n-1)}-1}.
\end{equation}
\end{proposition}

\begin{proposition}[Microlocal Kakeya--Nikodym estimates for the constant-curvature case]\label{prop main2}
Suppose we are in the constant-curvature setting, so all wave-packets are essentially supported in tubular neighborhoods of straight lines.
For every $\varepsilon>0$ and $\lambda\ge1$, and for all $2\,\frac{3n+2}{3n-2}\le p\le2\,\frac{n+1}{n-1}$,
\begin{equation}\label{eq:prop2}
    \|\mathcal T^\lambda f\|_{L^p(B^n_\lambda)}\lesssim_\varepsilon\ \lambda^\varepsilon\, \|f\|_{L^2(B^{n-1}_1)}^{2-\frac{2(n+1)}{p(n-1)}}\,
    \|f\|_{\widetilde\MKN(\lambda)}^{\frac{2(n+1)}{p(n-1)}-1}.
\end{equation}
\end{proposition}

\smallskip
It is clear that Theorem \ref{theo main2} follows directly from Proposition \ref{prop main1}, and Theorem \ref{theo main3} from Proposition \ref{prop main2}. Noting that
\[
\lambda^{\frac{n-1}{2}}\chi^\lambda f(x)=\chi_\lambda f(x/\lambda)+{\rm RapDec}(\lambda)\,\|f\|_{L^2},
\]
we conclude that Theorem \ref{theo main5} follows from Propositions \ref{prop main1} and \ref{prop main2}. Theorem \ref{theo main4} follows from Theorem \ref{theo main5} together with Lemma \ref{lem Orthogonal}, \ref{lem:WSI} and \ref{lem:WSA}. Indeed, apply Theorem \ref{theo main5} with $f=e_\lambda$ and use the reproducing property $\chi_\lambda e_\lambda=e_\lambda$ to obtain an $L^p$ bound for $e_\lambda$ in terms of the microlocal Kakeya--Nikodym norm of the input. By Lemma \ref{lem:WSI}, each microlocal piece $\mathcal Q^{\lambda,s}_{\theta,v}e_\lambda$ is essentially contained in a single tube $T^{s}_{\theta,v}$, and by Lemma \ref{lem Orthogonal} its mass is dominated by the $L^2$ mass of $e_\lambda$ on that tube. Taking the supremum over packets yields the standard comparison (see \cite{blair2015refined})
\[
\|e_\lambda\|_{\MKN(\lambda)}\lesssim\|e_\lambda\|_{\KN(\lambda)}.
\]
In the constant sectional curvature case, the same argument with Lemma \ref{lem:WSA} gives the anisotropic version. Therefore, Theorem \ref{theo main4} follows from Theorem \ref{theo main5}, and to prove all results in the paper, it suffices to establish Propositions \ref{prop main1} and \ref{prop main2}.

\medskip 

\subsection{Refined decoupling inequalities and some lemmas}

We now state the refined decoupling inequalities for our wave-packet decomposition.
The concept of decoupling inequalities was introduced by Wolff \cite{Wolff-decoupling}.
Variable-coefficient analogues of the Bourgain--Demeter decoupling theorem \cite{bourgain2015proof} were first obtained by Beltran, Hickman, and Sogge \cite{beltran2020variable} for the cone. Refined (multiplicity-sensitive) versions were proved in the translation-invariant setting for the paraboloid in \cite{guth2020falconer} and independently observed by Du and Zhang, and were extended to variable coefficients in \cite{iosevich2022microlocal}, which cover the class of $n\times n$ positive-definite (symmetric) Carleson--Sj\"olin phases; see also the appendix of \cite{gan2025local}. The translation-invariant statement implies its variable-coefficient counterpart by the standard induction-on-scales/parabolic-rescaling scheme: decoupling constants are multiplicative, so one propagates gains from small to large scales, while at small scales the variable-coefficient operators are well approximated by the translation-invariant model.

\begin{theorem}[Refined decoupling {\cite{guth2020falconer,iosevich2022microlocal}}]\label{refine}
Let $p=2\,\frac{n+1}{n-1}$, $1\le R\le\lambda$, and $s=R^{-1/2+\varepsilon_0}$. Let $\mathbb T$ be a set of wave-packets at scale $(R,s)$, and suppose
\[
f=\sum_{T\in\mathbb T} f_{T}
\quad\text{with}\quad
\|\mathcal T^\lambda f_{T}\|_{L^p(w_{B_R})}\ \text{comparable for all } T\in\mathbb T.
\]
Let $X\subset B_R$ be a union of $R^{1/2}$-balls such that each $R^{1/2}$-ball $Q\subset X$ meets at most $M$ tubes from $\mathbb T$. Then 
\[
\|\mathcal T^\lambda f\|_{L^p(X)}^p
\lesssim_\e R^\e M^{\frac{2}{n-1}}
\sum_{T\in \mathbb{T}}\|\mathcal T^\lambda f_{T}\|_{L^p(w_{B_R})}^p.
\]
Here $w_{B_R}$ is a nonnegative weight that is $\ge 1$ on $B_R$ and decays rapidly outside $B_R$.
\end{theorem}

\smallskip 

We collect several results for later use.
The first is a direct consequence of the classical $L^2\to L^2$ boundedness for oscillatory integral operators due to H\"ormander \cite{hormander1973oscillatory}. It can also be proved by hand via the wave-packet decomposition, at the expense of a harmless $R^{O(\e_0)}$ loss.

\begin{lemma}\label{hormander}
Let $1\leq R\leq \lambda$. Then
\begin{equation*}
  \|\mathcal{T}^\lambda f\|_{L^2(B_R)}^2 \lesssim R \|f\|_{L^2}^2.
\end{equation*}
\end{lemma}

The second result is proved by Stein in  \cite{stein1986oscillatory}.
\begin{lemma}
\label{Stein}
We have
\begin{equation}
\label{Stein-esti}
    \|\mathcal{T}^\lambda f\|_{L^{2\,\frac{n+1}{n-1}}}\lesssim \|f\|_{L^2}.
\end{equation}
\end{lemma}

\begin{lemma}\label{lem:l2}
Let $X = \bigcup Q$ be a union of $R^{1/2}$-balls, and let $f = \sum_{T \in \mathbb{T}} f_{T}$ be a sum of wave-packets at scale $R^{-1/2+\e_0}$. Suppose that for each $T \in \mathbb{T}$ there exists a shading $Y(T) \subset T$, consisting of $R^{1/2}$-balls in $X$, such that the number of $R^{1/2}$-balls contained in $Y(T)$ satisfies 
\[
\#\{ Q \subset Y(T) \} \lesssim \rho R^{1/2}.
\]
Then 
\begin{equation}\label{eq:l2}
    \int_X \Big| \sum_{T \in \mathbb{T}} \mathcal{T}^\lambda f_T \, \mathbf{1}_{Y(T)} \Big|^2 
    \lesssim R^{\e_0}\rho R \, \|f\|_{L^2}^2.
\end{equation}
\end{lemma}

\begin{proof}
Since $X = \bigcup Q$, we have
\[
\int_X \Big| \sum_{T \in \mathbb{T}} \mathcal{T}^\lambda f_T \, \mathbf{1}_{Y(T)} \Big|^2
\lesssim \sum_{Q \subset X} \Big| \sum_{T \in \mathbb{T}} \int_Q \mathcal{T}^\lambda f_T \, \mathbf{1}_{Y(T)} \Big|^2.
\]
Applying Lemma \ref{hormander} to each $Q$ at scale $R^{1/2}$, it follows that
\begin{equation}
\nonumber
\sum_{Q \subset X} \Big| \sum_{T \in \mathbb{T}} \int_Q \mathcal{T}^\lambda f_T \, \mathbf{1}_{Y(T)} \Big|^2
\lesssim R^{\e_0}R^{1/2} \sum_{Q \subset X} \Big\| \sum_{\substack{T \in \mathbb{T} \\ Y(T) \cap Q \neq \varnothing}} f_T \Big\|_{L^2}^2.
\end{equation}
Using orthogonality \eqref{eq:ortho1}, we further obtain
\begin{equation}
\nonumber
R^{1/2} \sum_{Q \subset X} \Big\| \sum_{\substack{T \in \mathbb{T} \\ Y(T) \cap Q \neq \varnothing}} f_T \Big\|_{L^2}^2
\lesssim R^{1/2} \sum_{T \in \mathbb{T}} \sum_{Q \subset Y(T)} \|f_T\|_{L^2}^2.
\end{equation}
Since the number of $R^{1/2}$-balls in $Y(T)$ is $\lesssim \rho R^{1/2}$, we conclude that
\[
R^{1/2} \sum_{T \in \mathbb{T}} \sum_{Q \subset Y(T)} \|f_T\|_{L^2}^2
\lesssim \rho R \sum_{T \in \mathbb{T}} \|f_T\|_{L^2}^2,
\]
which establishes \eqref{eq:l2}.
\end{proof}

\bigskip

\section{Proof of Proposition \ref{prop main1} }\label{sec general}
First, we fix notation for the remainder of the section. Because the incidence bounds are proved first, we work at unit scale: we consider the unit-length curves $\gamma_{\theta,v}\subset B_1$ rather than their rescalings to $B_\lambda$.

\begin{definition}[Admissible curves]
In the unified notation, for each choice of $\xi_\theta\in B^{n-1}_1(0)$ and $v\in B^{n-1}_1(0)$, there is an associated unit-length curve $\gamma_{\theta,v}\subset B^n_1(0)$.  
By ranging over all such pairs, we obtain the collection of \textbf{admissible curves}
\[
\Gamma:=\{\gamma_{\theta,v}:\ \xi_\theta,v\in B^{n-1}_1(0)\}.
\]
Given $0<\delta\ll1$, two curves $\gamma_{\theta_1,v_1},\gamma_{\theta_2,v_2}\in\Gamma$ are said to be \emph{$\delta$-separated} if
\[
|\xi_{\theta_1}-\xi_{\theta_2}|\ge C\delta,
\]
for a fixed absolute constant $C>0$.
\end{definition}

\begin{definition}[Refinement]
Let $E, F \subset \mathbb{R}^n$ be finite sets and let $c>0$.  
We say that $E$ is a \emph{$\gtrsim c$-refinement} of $F$ if 
$E \subset F$ and $\# E \gtrsim c \, \# F$.  
We say that $E$ is a \emph{$\gtrapprox c$-refinement} of $F$, or simply a \emph{refinement} of $F$, if $E \subset F$ and $\# E \gtrapprox c \, \# F$.
\end{definition}

\begin{definition}[Shading]\label{shading}
Let $\Gamma$ be a set of admissible curves in $\mathbb{R}^n$ and let $\delta \in (0,1)$.  
A \emph{shading} is a map $Y: \Gamma \to B_1(0)$ assigning to each curve $\gamma \in \Gamma$ a set 
\[
Y(\gamma) \subset N_\delta(\gamma) \cap B_1(0),
\]
where $Y(\gamma)$ is a union of $\delta$-balls in $\mathbb{R}^n$.  
We denote such a pair by $(\Gamma, Y)_\delta$ to emphasize the dependence on the scale parameter $\delta$.

Similarly, given a family of $\delta$-tubes
\[
\Gamma_\delta := \left\{ N_\delta(\gamma) : \gamma \in \Gamma \right\},
\]
a \emph{shading} is a map $Y: \Gamma_\delta  \to \mathbb{R}^n$ such that for each tube $T \in \Gamma_\delta$, the set $Y(T) \subset 2T$ is a union of $\delta$-balls in $\mathbb{R}^n$.
\end{definition}

\begin{definition}[Katz--Tao $(\delta,s,C)$-set]
Let $\delta \in (0,1)$ be small, and let $s \in (0,n]$.  
A finite set $E \subset \mathbb{R}^n$ is called a \emph{Katz--Tao $(\delta,s,C)$-set} (or simply a \emph{Katz--Tao $(\delta,s)$-set}) if
\[
\#\big(E \cap B_r(x)\big) \le C\left( \frac{r}{\delta} \right)^s,
\quad \forall\, x \in \mathbb{R}^n, \ \ r \in [\delta,1].
\]
\end{definition}

\begin{definition}[$\rho$-dense]
Let $\delta \in (0,1)$ and let $(\Gamma, Y)_\delta$ be a set of admissible curves with shading.  
We say that $Y$ is \emph{$\rho$-dense} if
\[
|Y(\gamma)| \geq \rho \, |N_\delta(\gamma)| \quad \text{for all } \gamma \in \Gamma.
\]
\end{definition}
\begin{definition}
\label{two-ends-def}
Let $\de\in(0,1)$ and let $(\Gamma,Y)_\delta$ be a set of admissible curves and shading.
Let $0<\e_2<\e_1<1$.
We say $Y$ is {\bf $(\e_1 ,\e_2, C)$-two-ends} if for all $\gamma \in \Gamma$ and all $\de\times\de^{\e_1}$-tubes $J\subset N_\de(\gamma)$, 
\begin{equation}
\nonumber
    |Y(\gamma)\cap J|\le C\de^{\e_2} |Y(\gamma)|.
\end{equation}
When the constant $C$ is not important in the context, we say $Y$ is {\bf $(\e_1 ,\e_2)$-two-ends}, or simply {\bf two-ends}. 
A similar definition applies to a single shading $Y(\gamma)$.
\end{definition}

\smallskip

To make use of Theorem \ref{refine}, we need the following incidence estimate to control the multiplicity factor $M$.
We remark that the exponent $1/2$ in the expression $(\delta^{n-1}\#\Gamma)^{1/2}$ below is crucial---it is the main reason we are able to put a considerable weight on the $L^2$ part of \eqref{eq:prop1}.

\begin{lemma}[Two-ends bush]\label{bush}
   Let $\delta \in (0,1)$. Let $(\Gamma, Y)_\delta$ be a set of $\delta$-separated curves  in $\mathbb{R}^n$  with an $(\varepsilon_1,\varepsilon_2)$-two-ends, $\rho$-dense shading. Then $|E_\Gamma|\gtrapprox \delta^{\varepsilon_1/2}\rho \delta^{\frac{n-1}{2}}(\delta^{n-1}\#\Gamma)^{1/2}$.
\end{lemma}
\begin{proof}
Define 
\[ E_{\Gamma}:=\bigcup_{\gamma \in \Gamma} N_\delta(\gamma).\]
Let $\mu=\sup_{x}\#\{\ell: x\in Y(\ell)\}$.
On the one hand,
\begin{equation}
\nonumber
    |E_\Gamma|\geq \mu^{-1}\sum_{\gamma \in \Gamma}|Y(\gamma)|\geq \mu^{-1}\rho (\de^{n-1}\# \Gamma).
\end{equation}
On the other hand, consider a single bush rooted at the point $x$ such that $\# \gamma(x)\sim  \mu$, where
\[
\gamma(x):=\{\gamma: x\in Y(\gamma)\}.
\]
Since the shading $Y(\gamma)$ is two-ends and $\rho$-dense, we have
\begin{equation}
\nonumber
    |E_\Gamma|\geq\big|\bigcup_{\gamma \in \gamma(x)}Y(\gamma)\big|\gtrsim \de^{\e_1} \mu\rho \de^{n-1}.
\end{equation}
The two estimates together give
\begin{equation}
\nonumber
    |E_\Gamma|\gtrapprox \de^{\e_1/2}\rho \delta^{\frac{n-1}{2}}(\delta^{n-1}\#\Gamma)^{1/2}. \qedhere
\end{equation}
\end{proof}
\begin{proposition}\label{pro1}
Let $\delta \in (0,1)$. Let $(\Gamma,Y)_\delta$ be a set of $\delta$-separated curves with an $(\varepsilon_1,\varepsilon_2)$-two-ends, $\rho$-dense shading. Take $\mu\sim \delta^{-\varepsilon_1}(\# \Gamma)^{1/2}$. Then there exists a set $E_\mu \subset E_{\Gamma}$ such that $\# \gamma(x)\lessapprox \mu$ for all $x\in E_\mu$ and 
  \[
  |E_\Gamma \backslash E_\mu|\leq \delta^{\varepsilon_1} |E_\Gamma|.
  \]
\end{proposition}
\begin{proof}
Let $ \eta \in [\mu, \delta^{-(n-1)}]$ be dyadic, and define
\[
E_\eta= \{ x \in E_\Gamma : \#\gamma(x) \sim \eta \}.
\]
Clearly,
\[
\eta\, |E_\eta| \lesssim \sum_{\gamma \in \Gamma} |Y(\gamma)|.
\]
We define 
\[
E'_\mu = \bigcup_{\eta \gtrsim \mu} E_\eta,
\quad E_\mu = E_\Gamma \setminus E'_\mu,
\]
so that
\begin{equation}\label{upper}
|E'_\mu|
= \sum_{\eta  \gtrsim \mu} \eta^{-1} \sum_{\gamma \in \Gamma } |Y(\gamma)|
\lesssim \mu^{-1} \sum_{\gamma  \in \Gamma} |Y(\gamma)|
\le\delta^{\varepsilon_1} \rho \delta^{n-1}(\#\Gamma)^\frac12.
\end{equation}

Now let $\Gamma' \subset \Gamma$ be a maximal $\delta$-separated set of admissible curves such that  
$|Y(\gamma')| \ge |Y(\gamma)|$ for all $\gamma'\in \Gamma'$ and $\gamma \in \Gamma$ with $\gamma,\gamma'$ are $\delta$-separated.  
Then
\[
\sum_{\gamma' \in \Gamma'} |Y(\gamma')| \gtrsim \sum_{\gamma \in \Gamma} |Y(\gamma)|.
\]
Applying Lemma \ref{bush} to $(\Gamma',Y)_\delta$ gives 
\begin{equation}\label{lower}
|E_\Gamma| \ge|E_{\Gamma'}|
\gtrsim\delta^{\varepsilon_1/2} \rho \delta^{n-1}(\#\Gamma)^\frac12.
\end{equation}

Combining \eqref{upper} and \eqref{lower} yields the desired estimate
\[
|E_\Gamma \setminus E_\mu| \lesssim \delta^{\varepsilon_1} |E_\Gamma|. \qedhere
\]
\end{proof}

\smallskip

Via interpolation with \eqref{Stein-esti}, to prove Proposition \ref{prop main1}, it suffices to show that
\begin{equation}
\label{goal-after-reduction}
    \|\mathcal{T}^\lambda f\|_{L^{p_n}(B_\lambda)}^{p_n}
    \lesssim \lambda^{\varepsilon} 
    \|f\|_{L^2}^{\frac{2(3n-1)}{3n-3}} 
\|f\|_{{\MKN(\lambda)}}^{\frac{4}{3n-3}},
\end{equation}
here  $p_n = 2\,\frac{3n+1}{3n-3}$.

To this end, we employ the induction-on-scales argument.
Let $1 \leq R \leq \lambda$ and define $S(R)$ to be the smallest constant such that
\begin{equation}
\nonumber
    \|\mathcal{T}^\lambda f\|_{L^{p_n}(B_R)}^{p_n}
    \leq  S(R) \,
    \|f\|_{L^2}^{\frac{2(3n-1)}{3n-3}} 
    \|f\|_{\MKN(R)}^{\frac{4}{3n-3}}.
\end{equation}
Our goal is to prove that, for all $\e>0$ and all $1 \leq R \leq \lambda$, we have
\begin{equation}\label{shac}
    S(R) \lesssim_\varepsilon R^\varepsilon.
\end{equation}
This would imply \eqref{goal-after-reduction} by taking $R=\lambda$, hence finishing the proof of Proposition \ref{prop main1}.
Fix $\varepsilon_0=\varepsilon^{1000}$ as in \eqref{eps_0}.
If $C_\varepsilon$ is chosen sufficiently large, \eqref{shac} is trivial for $1\le R\le 100$.

We now argue by induction on $R$. Assume that for all $\bar R$ with $1\le \bar R< R/2$,
\begin{equation}\label{eq:induct}
    S(\bar R)\lesssim_\varepsilon\bar R^{\varepsilon}.
\end{equation}
We will prove that the same bound holds at scale $R$.

\subsection{Reductions} 
Let $\mathbb{T}$ denote the collection of wave-packets $T$ at scale $s=R^{-1/2+\e_0}$ such that $T \cap B_R \neq \varnothing$.
By Lemma \ref{wpt-spt-lem}, we have
\begin{equation}
\label{reduction-1}
    \|\mathcal T^\lambda f\|_{L^{p_n}(B_R)}
    \lesssim \Big\|\sum_{T \in \mathbb{T}} \mathcal T^\lambda f_T \Big\|_{L^{p_n}(B_R)}
    + {\rm RapDec}(\lambda)\|f\|_{L^2}.
\end{equation}

For two dyadic numbers $\alpha_1,\alpha_2$, we define $\mathbb{T}_{\alpha_1,\alpha_2} \subset \mathbb{T}$ to be the subcollection of wave-packets such that for  each $\mathbb{T}_{\alpha_1,\alpha_2}$, $\|f_T\|_2\sim\alpha_1\|f\|_2$, and
\begin{equation*}
     \| \mathcal{T}^\lambda  f_T \|_{L^{q_n}(w_{B_R})} \sim \alpha_2 \|f\|_{L^2}, \hspace{.5cm}q_{n}=\frac{2(n+1)}{n-1}.
\end{equation*}
By a standard pigeonholing argument (see \cite[Section 5]{Wang-Wu}), either we can deduce \eqref{goal-after-reduction}, or there exists a pair $(\alpha_1,\alpha_2)$ for which
\begin{equation*}
    \Big\|\sum_{T \in \mathbb{T}} \mathcal T^\lambda f_{T}\Big\|_{L^{p_n}(B_R)}
    \lessapprox \Big\|\sum_{T \in \mathbb{T}_{\alpha_1,\alpha_2}} \mathcal T^\lambda f_{T}\Big\|_{L^{p_n}(B_R)}.
\end{equation*}

Applying a further dyadic pigeonholing, we can find a family of disjoint $R^{1/2}$-balls $Q$ such that
$\|\sum_{T\in \mathbb{T}_{\alpha_1,\alpha_2}}\mathcal T^\lambda f_T\|_{L^{p_n}(Q)}$ are comparable (up to multiplicative constants) across $Q$.  
Letting $X = \bigcup Q$, we then have
\begin{equation}
\label{reduction2}
    \Big\|\sum_{T \in \mathbb{T}_{\alpha_1,\alpha_2}} \mathcal T^\lambda f_T\Big\|_{L^{p_n}(B_R)}
    \lessapprox \Big\|\sum_{T \in \mathbb{T}_{\alpha_1,\alpha_2}} \mathcal T^\lambda f_T\Big\|_{L^{p_n}(X)}.
\end{equation}
For convenience, and without ambiguity, we henceforth write $\mathbb{T}$ in place of $\mathbb{T}_{\alpha_1,\alpha_2}$.

\subsection{Two-ends reductions}
For each $T \in \mathbb{T}$, we partition $T$ into sub-tubes $\mathcal{J} = \{J\}$ of length $R^{1 - \varepsilon^2}$. 
Each $J$ is of the form \[\{T\cap\{x_n\in[mR^{1-\e^2},(m+1)R^{1-\e^2})\}:  m\in\Z\cap[-2R^{\e^2}, 2R^{\e^2}]\}.\]
Then we decompose the set $\mathcal{J}(T)$ as a disjoint union $\mathcal{J}(T) = \bigcup_\rho \mathcal{J}_\rho(T)$, where $\rho \leq 1$ is a dyadic number and, for each $J \in \mathcal{J}_\rho(T)$, we have
\[
\{Q\subset X \text{ an $R^{1/2}$-ball}: Q\cap J\not=\varnothing\} \sim \rho  R^{1/2}.
\]
Hence,
\begin{equation}
\nonumber
    \sum_{T \in \mathbb{T}} \mathcal T^\lambda f_T = \sum_{\rho} \sum_{T \in \mathbb{T}} \sum_{J \in \mathcal{J}_\rho(T)} \mathcal T^\lambda f_T \mathbf{1}_{J}.
\end{equation}
For each cube $Q \subset X$, by pigeonholing, there exists a dyadic number $\rho(Q)$ such that
\begin{equation}
\nonumber
    \Big\| \sum_{T \in \mathbb{T}} \mathcal T^\lambda f_T \Big\|_{L^{p_n}(Q)}^{p_n} 
    \lessapprox 
    \int_Q \Big| \sum_{T \in \mathbb{T}} \sum_{J \in \mathcal{J}_{\rho(Q)}(T)} \mathcal T^\lambda f_T \mathbf{1}_{J} \Big|^{p_n}.
\end{equation}
Note that $\|\mathcal T^\lambda f_T\|_{L^{p_n}(Q)}$ is comparable up to a constant multiple for all $Q \subset X$. By dyadic pigeonholing over the collection
\[
\Big\{ \Big( \rho(Q), \Big\| \sum_{T \in \mathbb{T}} \sum_{J \in \mathcal{J}_{\rho(Q)}(T)} \mathcal T^\lambda f_T \mathbf{1}_{J} \Big\|_{L^{p_n}(Q)} \Big) : Q \subset X \Big\},
\]
we obtain a uniform dyadic number $\rho$ and a refinement $X_1 \subseteq X$ such that for all $Q \subset X_1$, we have $\rho(Q) = \rho$, and
\[
\Big\| \sum_{T \in \mathbb{T}} \sum_{J \in \mathcal{J}_\rho(T)} \mathcal T^\lambda f_T \mathbf{1}_{J} \Big\|_{L^{p_n}(Q)}
\]
is the same up to a constant multiple. Therefore,
\begin{equation}
\nonumber
    \Big\| \sum_{T \in \mathbb{T}}  \mathcal T^\lambda f_T \Big\|_{L^{p_n}(X)} 
    \lessapprox 
    \Big\| \sum_{T \in \mathbb{T}} \sum_{J \in \mathcal{J}_\rho(T)} \mathcal T^\lambda f_T \mathbf{1}_{J} \Big\|_{L^{p_n}(X_1)}.
\end{equation}

Next, we partition $\mathbb{T} = \bigcup_{\beta} \mathbb{T}_\beta$, where $\beta \in [1, R^{\varepsilon^2}]$ is a dyadic number such that for all $T \in \mathbb{T}_\beta$, we have $\# \mathcal{J}_\rho(T) \sim \beta$. Consequently,
\[
\sum_{T \in \mathbb{T}} \sum_{J \in \mathcal{J}_\rho(T)} \mathcal T^\lambda f_T \mathbf{1}_{J}
=
\sum_{\beta} \sum_{T \in \mathbb{T}_\beta} \sum_{J \in \mathcal{J}_\rho (T)} \mathcal T^\lambda f_T \mathbf{1}_{J}.
\]
For each $Q \in X_1$, by pigeonholing again, there exists a dyadic number $\beta(Q)$ such that
\[
\int_Q \Big| \sum_{T \in \mathbb{T}} \sum_{J \in \mathcal{J}_\rho(T)} \mathcal T^\lambda f_T \mathbf{1}_{J} \Big|^{p_n}
\lessapprox
\int_Q \Big| \sum_{T \in \mathbb{T}_{\beta(Q)}} \sum_{J \in \mathcal{J}_\rho (T)} \mathcal T^\lambda f_T \mathbf{1}_{J} \Big|^{p_n}.
\]
Since the $L^{p_n}(Q)$ norms above are all comparable for $Q \subset X_1$, we may apply dyadic pigeonholing once more to obtain a fixed dyadic number $\beta$ and a refinement $X_2 \subseteq X_1$ such that, for all $Q \subset X_2$, we have:
\begin{itemize}
    \item $\beta(Q) = \beta$,
    \item $\Big\| \sum_{T \in \mathbb{T}_\beta} \sum_{J \in \mathcal{J}_\rho (T)} \mathcal T^\lambda f_T \mathbf{1}_{J} \Big\|_{L^{p_n}(Q)}$ is the same up to a constant multiple.
\end{itemize}
Therefore, we conclude:
\begin{equation}
\label{X-X-2}
    \Big\| \sum_{T \in \mathbb{T}} \mathcal T^\lambda f_T \Big\|_{L^{p_n}(X)}^{p_n} 
    \lessapprox 
    \int_{X_2} \Big| \sum_{T \in \mathbb{T}_\beta} \sum_{J \in \mathcal{J}_\rho (T)} \mathcal T^\lambda f_T \mathbf{1}_{J} \Big|^{p_n}.
\end{equation}
Moreover, since $X_2$ is a refinement of $X_1$, we have
\[|X_2|\gtrapprox|X|.\]

\subsection{The non-two-ends scenario}
Suppose $\beta\leq R^{\varepsilon^4}$.
Let $B_k$ be a family of finitely overlapping $R^{1-\varepsilon^2}$-balls that cover $B_R$. 
Then
\begin{equation*}
    \int_{X_2} \Big|\sum_{T\in \mathbb{T}_\beta}\sum_{J\in \mathcal{J}_\rho (T)}\mathcal T^\lambda f_{T}\mathbf{1}_{J}\Big|^{p_n}\lesssim \sum_k \int_{X_2\cap B_k}\Big|\sum_{T\in \mathbb{T}_\beta}\sum_{J \in \mathcal{J}_\rho(T)}\mathcal T^\lambda f_{T}\mathbf{1}_J\Big|^{p_n}.
\end{equation*}
For each $B_k$, define
\begin{equation*}
    f_k=\sum_{\substack{T\in \mathbb{T}_\beta \;\text{such that}\\ \exists J \in \mathcal{J}_\rho (T), J\cap B_k\neq \varnothing}} f_{T}.
\end{equation*}
Hence, up to a rapidly decaying term, we have
\begin{equation}
\label{non-two-ends-after}
    \int_{X_2\cap B_k}\Big|\sum_{T\in \mathbb{T}_\beta}\sum_{J\in \mathcal{J}_\rho(T)}\mathcal T^\lambda f_T \mathbf{1}_J\Big|^{p_n}\lesssim  \int_{ B_k}\Big|\sum_{T\in \mathbb{T}_\beta}\sum_{J\in \mathcal{J}_\rho (T)}\mathcal T^\lambda f_T \mathbf{1}_J\Big|^{p_n}\lesssim  \int_{ B_k}\Big|\mathcal T^\lambda f_k\Big|^{p_n}.
\end{equation}
Note that for  each $T$, there are $\lessapprox R^{\varepsilon^4}$ many $B_k$ such that there exists $ J \in \mathcal{J}_\rho(T), J\cap B_k\neq \varnothing$. Thus,
\begin{equation}
\nonumber
    \sum_{k}\|f_k\|_2^2\lesssim R^{\varepsilon^4}\|f\|_2^2.
\end{equation}
 By invoking our induction hypothesis \eqref{eq:induct}  on each $R^{1-\varepsilon^2}$-ball $B_k$, we obtain that 
\begin{equation}
\nonumber
    \|\mathcal T^\lambda f_k\|_{L^{p_n}(B_k)}^{p_n}\lesssim_\varepsilon R^{(1-\varepsilon^2)\varepsilon}\|f_k\|_2^{2\,\frac{3n-1}{3n-3}} \|f_k\|_{\MKN(R^{1-\varepsilon^2})}^{\frac{4}{3n-3}}.
\end{equation}
Note that $\|f_k\|_{\MKN(R^{1-\varepsilon^2})}$ corresponds to the microlcal Kakeya--Nikodym  norm at the scale $R^{1-\varepsilon^2}$.
It is straightforward  to verify 
\begin{equation*}
  \|f_k\|_{L^2}\lesssim \|f\|_{L^2},\; \; 
  \|f_k\|_{\MKN(R^{1-\varepsilon^2})}\lesssim \|f\|_{\MKN(R)}. 
\end{equation*}
Sum over all $B_k$ so that
\begin{equation}
\nonumber
    \int_{X_2}\Big|\sum_{T\in \mathbb{T}_\beta}\sum_{J\in \mathcal{J}_\rho(T)} \mathcal T^\lambda f_T \mathbf{1}_{J}\Big|^{p_n}\lesssim \sum_{k}C_\varepsilon R^{(1-\varepsilon^2)\varepsilon}\|f_k\|_{L^2}^{2\,\frac{3n-1}{3n-3}}\|f\|_{\MKN(R)}^{\frac{4}{3n-3}}.
\end{equation}
Finally, by \eqref{reduction-1}, \eqref{reduction2}, and \eqref{X-X-2}, we have
\begin{equation}
\nonumber
    \|\mathcal T^\lambda f\|_{L^{p_n}(B_R)}^{p_n}\lesssim R^{(-\varepsilon^3+\varepsilon^4)\varepsilon}C_\varepsilon R^\varepsilon \|f\|_{L^2}^{2\,\frac{3n-1}{3n-3}}\|f\|_{\MKN(R)}^{\frac{4}{3n-3}}.
\end{equation}
This closes the induction, and thus completes the discussion of the non-two-ends case.

\subsection{The two-ends scenario} 
Suppose $\beta \in [R^{\varepsilon^4}, R^{\varepsilon^2}]$. For each $T \in \mathbb{T}_\beta$, consider the shading
\[
Y(T) = \bigcup_{J \in \mathcal{J}_{\rho}(T)} \{Q\subset X \text{ an $R^{1/2}$-ball}: Q\cap J\not=\varnothing\}.
\]
Then $Y$ is an $(\varepsilon^2, \varepsilon^4)$ two-ends, $\rho\beta$-dense shading, in the sense that 
\begin{equation}\label{shadingc}
    |Y(T)| \geq \rho \beta |T|.
\end{equation}
Up to a $R^{-1}$-dilation, this corresponds to the Definition \ref{shading}  with $\delta = R^{-1/2+\e_0}$.

\smallskip

On the one hand, take
\begin{equation}
\label{mu-1}
\mu = R^{2\varepsilon^2}  (\#\mathbb{T})^{1/2}.
\end{equation}
Note that the configuration $(\mathbb{T}_\beta, Y)$ is $(\varepsilon^2, \varepsilon^4)$-two-ends, and the radius of each $T$ is $R^{1/2+\varepsilon_0}$.  
After accepting a loss of $R^{O(\varepsilon_0)}$, we may apply Proposition \ref{pro1} to the $R^{-1}$-dilate of $(\mathbb{T}_\beta, Y)$, obtaining a set $X_3 \subset X$ with
\begin{equation}
\label{mu-upper-bound}
    \sup_{Q \subset X_3} \#\{ T \in \mathbb{T}_\beta : Q\subset Y(T) \} \lesssim R^{O(\varepsilon_0)} \mu
\end{equation}
and $|X \setminus X_3| \le R^{-\varepsilon^2} |X| \lessapprox R^{-\varepsilon^2} |X_2|$.

Let $X_4 := X_2 \cap X_3$.  
Then $|X_4| \gtrsim |X_2|$ and $X_4 \subset X_3$.  
Since
\[
\Big\| \sum_{T \in \mathbb{T}_\beta} \sum_{J \in \mathcal{J}_{\rho}(T)} \mathcal{T}^\la  f_T \mathbf{1}_J \Big\|_{L^{p_n}(Q)}
\]
is comparable for all $Q \subset X_2$, we have
\begin{equation}
\label{X-2-X-4}
    \int_{X_2} \Big| \sum_{T \in \mathbb{T}_\beta} \sum_{J \in \mathcal{J}_\rho(T)} 
    \mathcal{T}^\la f_T \mathbf{1}_J \Big|^{p_n}
    \lessapprox 
    \int_{X_4} \Big| \sum_{T \in \mathbb{T}_\beta} \sum_{J \in \mathcal{J}_\rho(T)} 
\mathcal{T}^\la f_T \mathbf{1}_J \Big|^{p_n}.
\end{equation}
Recall that $\{B_k\}$ denotes a cover of $B_R$ by $R^{1-\varepsilon^2}$-balls.  
For each $B_k$, set
\[
\mathbb{T}_{\beta,k} := 
\{ T \in \mathbb{T}_\beta : \exists J \in \mathcal{J}_{\rho}(T),\ J \cap B_k \neq \varnothing \}.
\]
Thus, we have
\begin{align*}
&\int_{X_4} \Big| \sum_{T \in \mathbb{T}_\beta} 
\sum_{J \in \mathcal{J}_{\rho}(T)} \mathcal{T}^\la  f_T \mathbf{1}_J \Big|^{p_n}
\lessapprox \sum_k \int_{X_4 \cap B_k} 
\Big| \sum_{T \in \mathbb{T}_\beta} \sum_{J \in \mathcal{J}_\rho{(T)}} 
\mathcal{T}^\la  f_T \mathbf{1}_J \Big|^{p_n}  \\
&\lessapprox \sum_k \int_{X_4 \cap B_k} 
\Big| \sum_{T \in \mathbb{T}_{\beta,k}} \mathcal{T}^\la  f_T \Big|^{p_n} \notag \lessapprox R^{O(\varepsilon^2)} \sup_k 
\int_{X_4 \cap B_k} \Big| \sum_{T \in \mathbb{T}_{\beta,k}} \mathcal{T}^\la  f_T \Big|^{p_n}. 
\end{align*}
In the second inequality, we use the fact that $\mathbf{1}_J$ is identically 1 or 0 on $B_k$, which holds by assuming that $X_2$ does not intersect the horizontal lines $\{x_n=mR^{1-\e^2}\}, m\in\Z\cap [-2R^{\e^2}, 2R^{\e^2}]$, a condition that can be achieved using the triangle inequality.

For each $Q \subset X_4 \cap B_k = X_3 \cap B_k$, note that
\begin{equation*}
\#\{ T \in \mathbb{T}_\beta : Y(T) \cap Q \neq \varnothing \}
= \#\{ T \in \mathbb{T}_{\beta,k} : T \cap Q \neq \varnothing \}.
\end{equation*}
Apply Theorem \ref{refine} and \eqref{mu-upper-bound} so that
\begin{equation}
\label{after-dec}
\int_{X_4 \cap B_k} \Big| \sum_{T \in \mathbb{T}_{\beta,k}} \mathcal{T}^\la f_T \Big|^{q_n}
\lessapprox R^{O(\varepsilon_0)} \mu^{\frac{2}{n-1}}
\sum_{T \in \mathbb{T}} \| \mathcal{T}^\la  f_T \|_{L^{q_n}(w_{B_R})}^{q_n}.
\end{equation}

Since all $\|f_{T}\|_{L^2}$ are comparable, by Lemma \ref{Stein}, we have
\begin{equation}
\nonumber
    \sum_{T \in \mathbb{T}} \| \mathcal{T}^\la  f_T \|_{L^{q_n}(w_{B_R})}^{q_n}\lesssim \sum_{T \in \mathbb{T}}\|f_T\|_2^{q_n}\lesssim (\#\ZT)^{-\frac{2}{n-1}}\|f\|_2^{q_n}.
\end{equation}
Substituting this into \eqref{after-dec} and applying Lemma \ref{KN-L2-lem}, we obtain
\begin{align}
\label{refined-strichartz}
    \int_{X_4 \cap B_k} \Big| \sum_{T \in \mathbb{T}_{\beta,k}} \mathcal{T}^\la f_T \Big|^{q_n}&\lessapprox R^{O(\varepsilon_0)}\Big(\frac{\mu}{\#\mathbb{T}}\Big)^{\frac{2}{n-1}}\|f\|_2^{q_n}\\ \nonumber
    &\lessapprox R^{O(\varepsilon_0)}\frac{\mu^\frac{2}{n-1}}{(\#\mathbb{T})^{\frac{1}{n-1}}}R^{-1/2}\|f\|_2^{\frac{2n}{n-1}}\|f\|_{\MKN(R)}^{\frac{2}{n-1}}.
\end{align}
Since $\mu = R^{2\e^2}(\#\mathbb{T})^{1/2}$, it follows that
\begin{equation}
\label{L-qn}
    \int_{X_4 \cap B_k} \Big| \sum_{T \in \mathbb{T}_{\beta,k}} \mathcal{T}^\la f_T \Big|^{q_n}\lessapprox R^{O(\varepsilon_0+\e^2)}R^{-1/2}\|f\|_2^{\frac{2n}{n-1}}\|f\|_{\MKN(R)}^{\frac{2}{n-1}}.
\end{equation}

\smallskip 

On the other hand, by $L^2$ orthogonality, Lemma \ref{lem:l2}, and \eqref{shadingc}, we have 
\begin{equation}
\label{l2}
   \int_{X_4 \cap B_k} \Big| \sum_{T \in \mathbb{T}_{\beta,k}} \mathcal{T}^\la f_T \Big|^{2}\lesssim\int_{X_4 \cap B_k} | \mathcal{T}^\la f|^{2}\lesssim\|\mathcal T^\lambda f\|_{L^{2}(X_2)}^{2}\lesssim R^{O(\e_0)}\rho  \beta R \|f\|_2^2.
\end{equation} 
Recall that $\e_0=\e^{1000}$ and $\rho\be\leq R^{\e^2}$.
Interpolating between \eqref{L-qn} and \eqref{l2} via H\"older's inequality and putting it back to \eqref{X-2-X-4}, and then \eqref{X-X-2}, \eqref{reduction2}, and \eqref{reduction-1}, we finally obtain
\begin{equation}
   \|\mathcal{T}^\lambda f\|_{L^{p_n}(B_R)}^{p_n}
    \leq  R^{O(\e^2)} 
    \|f\|_{L^2}^{\frac{2(3n-1)}{3n-3}} 
    \|f\|_{\MKN(R)}^{\frac{4}{3n-3}}.
\end{equation} 
This closes the induction and thus completes the discussion of the two-ends case.

\bigskip

\section{Proof of Proposition \ref{prop main2}}\label{sec constantcurv}
In this section, we consider manifolds with constant sectional curvature. 
It was shown in \cite{gao2025curved} that the associated Kakeya curves can be straightened. Therefore, it suffices to consider the incidence problem in the Euclidean case.

To establish the incidence result we need, we will use the following two-ends Furstenberg estimate proved in \cite[Theorem 2.1]{Wang-Wu}.
\begin{theorem}
\label{two-ends-furstenberg}
Let $\de\in(0,1)$.
Let $(\ZT,Y)_\de$ be a set of directional $\de$-separated tubes in $\ZR^2$ with an $(\e_1, \e_2)$-two-ends, $\rho$-dense shading. Then for any $\e>0$,
\begin{equation}
\label{eq:thm2.1}
    \Big|\bigcup_{T\in\ZT}Y(T)\Big|\geq c_{\e,\e_2}\de^{\e}\de^{\e_1/2} \rho^{1/2}\sum_{T\in \ZT}|Y(T)|.
\end{equation}
\end{theorem}

We use Theorem \ref{two-ends-furstenberg} and Wolff's hairbrush \cite{Wolff-Kakeya} to prove the following lemma.

\begin{lemma}
\label{incidence-lem}
Let $(\ZT,Y)_\de$ be a set of tubes and shadings in $\ZR^n$, and let $m\geq1$. 
Suppose $\ZT$ obeys Wolff's axiom with error $m$, and $Y$ is an $(\e_1,\e_2)$-two-ends, $\rho$-dense shading.
Let $E=\bigcup_{T\in\ZT}Y(T)$.
For all $x\in E$, define $\ZT(x)=\{T\in\ZT:x\in Y(T)\}$.
Then for all $\e>0$,
\begin{equation}
\label{hairbrush-esti}
    |E|\geq c_\e\de^\e \de^{O(\e_1)}m^{-n/(2n-2)}\rho^{7/4}\de^{(n-2)/2}(\de^{n-1}\#\ZT)^{n/(2n-2)}.
\end{equation}
In particular, by taking $\mu=\de^{-O(\e_1)}m^{n/(2n-2)}\rho^{-3/4}\de^{-(n-2)/2}(\de^{n-1}\#\ZT)^{(n-2)/(2n-2)}$, there exists a set $E_\mu\subset E$ such that $\# \ZT(x)\lessapprox \mu$ for all $x\in E_\mu$, and
\begin{equation}
\label{exceptional-set}
    |E\setminus E_\mu|\leq \de^{\e_1}|E|.
\end{equation}
\end{lemma}

\begin{proof}
Let $\Sigma=\{\sigma\}$ be a set of $\de^{\e^2}$-caps that forms a finitely overlapping cover of $\ZS^{n-1}$.
For each $\sigma\in\Sigma$, define $$\ZT_\si(x)=\{T\in\ZT(x):\text{ the direction of $T$ lies in $\si$}\}.$$

Partition $\ZT(x)=\sqcup_\nu\ZT^\nu(x)$ so that for each dyadic $\nu$, there exists a set of  $\Sigma_\nu(x)\subset\Sigma$ such that $\#\ZT_\si(x)\sim\nu$ for all $\si\in\Si_\nu(x)$, and $\ZT^\nu(x)=\sqcup_{\si\in\Si_\nu(x)}\ZT_\si(x)$.
By pigeonholing, there exists a $\nu$ such that $\#\ZT(x)\gtrapprox\#\ZT^\nu(x)$.
By dyadic pigeonholing again, there exists a $\be$ and a set $E_1\subset E$ such that 
\begin{enumerate}
    \item $|E_1|\gtrapprox |E|$.
    \item $\#\ZT^\nu(x)\sim\be$ for all $x\in E_1$.
\end{enumerate}

For each $T\in\ZT$, define a new shading $$Y_1(T)=\{x\in Y(T):T\in\ZT_\nu(x)\text{ and }\#\ZT^\nu(x)\sim\be\}.$$
Then $(\ZT,Y_1)_\de$ is a refinement of $(\ZT,Y)_\de$.
Let $(\ZT_1,Y_1)_\de$ be a refinement of $(\ZT,Y_1)_\de$ so that $|Y_1(T)|\gtrapprox|Y(T)|$ for all $T\in\ZT_1$.
We consider two separate cases.

{\bf Case 1: $\be\leq \de^{-\e^4}$.}
Let $\bar\ZT$ be the family of $\de^{\e^{2}}$-tubes such that each $\de$-tube in $\ZT$ is contained in some $\bar T\in\bar\ZT$.
For each $\bar T\in\bar \ZT$, let $\ZT_1(\bar T)=\{T\in\ZT_1: T\subset \bar T\}$.
Let $E_{\bar T}=\cup_{T\in\ZT_1(\bar T)}Y_1(T)$.
Thus, by induction,
\begin{align*}
    |E|& \gtrsim \be^{-1}\sum_{\bar T}|E_{\bar T}|\\
    &\geq \de^{\e^4}\sum_{\bar T}c_\e\de^{\e(1-\e^2)} \de^{O(\e_1)}m^{-n/(2n-2)}\rho^{7/4}\de^{(n-2)/2}(\de^{n-1}\#\ZT_1(\bar T))^{n/(2n-2)}\\
    &\geq\de^{\e^4}c_\e\de^{\e(1-\e^2)} \de^{O(\e_1)}m^{-n/(2n-2)}\rho^{7/4}\de^{(n-2)/2}(\de^{n-1}\#\ZT_1)^{n/(2n-2)}.
\end{align*}
Since $(Y_1,\ZT_1)_\de$ is a refinement of $(Y,\ZT)_\de$, we have $\#\ZT_1\gtrapprox\#\ZT$.
Hence,
\begin{equation}
\nonumber
    |E|\gtrapprox \de^{\e^4-\e^3}\cdot c_\e\de^\e \de^{O(\e_1)}m^{-n/(2n-2)}\rho^{7/4}\de^{(n-2)/2}(\de^{n-1}\#\ZT)^{n/(2n-2)}.
\end{equation}
This gives what we want.

{\bf Case 2: $\be\geq \de^{-\e^4}$.}
Let $\mu=\nu\be$.
On the one hand, we have
\begin{equation}
\label{estimate-1}
    |E|\gtrsim\mu^{-1}\sum_{T\in\ZT_1}|Y_1(T)|\gtrapprox\mu^{-1}\rho(\de^{n-1}\#\ZT).
\end{equation}

On the other hand, consider a single hairbrush rooted at a tube $T_0$.
Let $$\ZH(T_0)=\{T\in\ZT:Y(T)\cap Y_1(T_0)\not=\varnothing\text{, and }\angle(T,T_0)\gtrsim\de^{\e^2}\}.$$
Since $|Y_1(T_0)|\gtrapprox\rho\de^{n-1}$ and since $\#\ZT_\nu(x)\sim\be\geq \de^{-\e^4}$ for all $x\in E_1$, we have $\#\ZH(T_0)\gtrapprox \mu\rho\de^{-1}$, 
For each $T\in\ZH(T_0)$, define a new shading $Y'(T)=Y(T)\setminus N_{\de^{1-10(\e_1+\e^2)}}(T_0)$.
Note that $Y'$ is still $(\e_1,\e_2)$-two-ends, as $Y$ is $(\e_1,\e_2)$-two-ends.

Let $\cp$ be the set of $\de$-separated planes that intersect the coreline of $T_0$.
For each $P\in\cp$, let $$\ZT(P)=\{T\in\ZH(T_0):T\subset N_\de(P)\}.$$
Since $$\{N_\de(P)\setminus N_{\de^{1-10(\e_1+\e^2)}}(T_0):P\in\cp\}$$ is $O(\de^{10(\e_1+\e^2)})$-overlapping, we have
\begin{equation}
\label{hairbrush}
    |E|\gtrsim \Big|\bigcup_{T\in\ZH(T_0)}Y(T)\Big|\gtrsim \de^{10(\e_1+\e^2)}\sum_{P\in\cp}\Big|\bigcup_{T\in\ZT(P)}Y'(T)\Big|.
\end{equation}
For each $P\in\cp$, let $\ZT'(P)$ be a random sample of $\ZT(P)$ with probability $m^{-1}$, so $\#\ZT'(P)\gtrapprox m^{-1}\#\ZT(P)$.
Since $\ZT$ obeys Wolff's axiom with error $m$, $\ZT(P)$ obeys Wolff's axiom with error $m$, yielding that $\ZT'(P)$ obeys Wolff's axiom with error $1$ (see \cite[Lemma 1.6]{Wang-Wu} for a similar probabilistic argument).
In other words, $\ZT'(P)$ is a Katz--Tao $(\d,1)$-set.
By Theorem \ref{two-ends-furstenberg}, we have
\begin{equation}
\nonumber
    \Big|\bigcup_{T\in\ZT(P)}Y'(T)\Big|\gtrsim\Big|\bigcup_{T\in\ZT'(P)}Y'(T)\Big|\gtrsim \rho^{3/2}(\de^{n-1}\#\ZT'(P))\gtrapprox m^{-1}\rho^{3/2}(\de^{n-1}\#\ZT(P)).
\end{equation}
Put this back to \eqref{hairbrush} so that 
\begin{align}
\label{estimate-2}
    |E|&\gtrsim\de^{10(\e_1+\e^2)}\sum_{P\in\cp}m^{-1}\rho^{3/2}(\de^{n-1}\#\ZT(P))\gtrapprox\de^{O(\e_1+\e^2)}m^{-1}\rho^{3/2}(\de^{n-1}\#\ZH(T_0))\\ \nonumber
    &\gtrapprox \de^{O(\e_1+\e^2)}m^{-1}\rho^{3/2}(\mu\rho\de^{n-2}).
\end{align}

The two estimates \eqref{estimate-1} and \eqref{estimate-2} imply
\begin{align*}
    |E|&\gtrsim \de^{(\e_1+\e^2)}m^{-1/2}\rho^{7/4}\de^{(n-2)/2}(\de^{n-1}\#\ZT)^{1/2}\\
    &\geq  c_\e\de^\e \de^{O(\e_1)}m^{-n/(2n-2)}\rho^{7/4}\de^{(n-2)/2}(\de^{n-1}\#\ZT)^{n/(2n-2)}.
\end{align*}
In the last inequality, we use the estimate $\de^{n-1}\#\ZT\lesssim m$. 

\smallskip

It is standard to deduce \eqref{exceptional-set} from \eqref{hairbrush-esti}.
See, for example, \cite[Corollary 3.4]{Wang-Wu}.
\qedhere

\end{proof}

\smallskip

\subsection{Proof of  Proposition \ref{prop main2} }
Let $p_n=2\,\frac{3n+2}{3n-2}$.
Similar to the argument in the previous section, we define $S(R)$ to be the smallest constant such that 
\begin{equation}
\nonumber
    \|\mathcal{T}^\lambda f\|_{L^{p_n}(B_R)}
    \leq  S(R) \,
    \|f\|_{L^2}^{2-\frac{2(n+1)}{p_n(n-1)}}
    \|f\|_{\wt\MKN(R)}^{\frac{2(n+1)}{p_n(n-1)}-1}.
\end{equation}
Our goal is to prove that, for all $\e>0$ and all $1 \leq R \leq \lambda$, we have
\begin{equation}
\nonumber
    S(R) \lesssim_\varepsilon R^\varepsilon.
\end{equation}
Take $\e_0=\e^{1000}$, where $\e_0$ is the fixed parameter \eqref{eps_0} in the wave-packet decomposition.

\smallskip 

Starting from \eqref{reduction-1}, we follow the induction argument in the previous section until \eqref{mu-1}.
In place of \eqref{mu-1}, choose another $\mu$ as 
\begin{equation}\label{mu}
    \mu = R^{O(\e^2)}\rho^{-3/4}R^{\frac{n-2}{4}}m^{\frac{n}{2n-2}}(R^{-\frac{n-1}{2}}\#\ZT)^{\frac{n-2}{2n-2}}.
\end{equation}
Then, we follow the remaining argument until \eqref{refined-strichartz} and have
\begin{align*}
    \int_{X_4 \cap B_k} \Big| \sum_{T \in \mathbb{T}_{\beta,k}} \mathcal{T}^\la f_T \Big|^{q_n}&\lessapprox R^{O(\varepsilon_0)}\Big(\frac{\mu}{\#\mathbb{T}}\Big)^{\frac{2}{n-1}}\|f\|_2^{q_n}.
\end{align*}
By Lemma \ref{KN-L2-lem2} and \eqref{mu}, we have the following estimate in place of \eqref{L-qn}
\begin{align}
\nonumber
   \int_{X_4 \cap B_k} \Big| \sum_{T \in \mathbb{T}_{\beta,k}} \mathcal{T}^\la f_T \Big|^{q_n}\lessapprox &\,
    R^{O(\varepsilon_0)}(\rho^{-3/4}R^{\frac{n-2}{4}}m^{\frac{n}{2n-2}}(R^{-\frac{n-1}{2}}\#\ZT)^{\frac{n-2}{2n-2}})^{\frac{2}{n-1}} (\#\ZT)^{-\frac{2}{n-1}}\\ \nonumber&
    \cdot({\#\ZT}^{1/2}m^{-1/2}R^{-\frac{n-1}{4}})^{\frac{2n}{(n-1)^2}}\|f\|_{L^2}^{\frac{2(n^2-n-1)}{(n-1)^2}}\|f\|_{\widetilde{\MKN}(R)}^{\frac{2n}{(n-1)^2}}\\[1ex] \label{L-qn-2}
    \lessapprox & \, R^{O(\varepsilon_0+\e^2)}\rho^{-\frac{3}{2(n-1)}} R^{-\frac{n}{2(n-1)}}\|f\|_{L^2}^{\frac{2(n^2 - n - 1)}{(n-1)^2}} \|f\|_{\widetilde{\MKN}(R)}^{\frac{2n}{(n-1)^2}}.
\end{align}

On the other hand, similar to \eqref{l2}, we have 
\begin{equation}
\label{l2-2}
   \int_{X_4 \cap B_k} \Big| \sum_{T \in \mathbb{T}_{\beta,k}} \mathcal{T}^\la f_T \Big|^{2}\lesssim\|\mathcal T^\lambda f\|_{L^{2}(X_2)}^{2}\lesssim R^{O(\e_0)}\rho  \beta R \|f\|_2^2.
\end{equation}
As in the previous section, we interpolate between \eqref{L-qn-2} and \eqref{l2-2} via H\"older's inequality to conclude
\begin{equation}
\nonumber
  \|\mathcal{T}^\lambda f\|_{L^{p_n}(B_R)}
    \lesssim R^{O(\e^2)} \,
    \|f\|_{L^2}^{2-\frac{2(n+1)}{p_n(n-1)}}
    \|f\|_{\wt\MKN(R)}^{\frac{2(n+1)}{p_n(n-1)}-1}.
\end{equation} 
This closes the induction.
{
\section{Sharpness examples}
\label{sec example}

\smallskip

Finally, we prove the sharpness part of Theorem \ref{theo main1}. 
In odd dimensions, this is the familiar Kakeya compression example for positive-definite H\"ormander operators \cite{minicozzi1997negative,GHI}. In even dimensions, we modify the compression examples from \cite{minicozzi1997negative,GHI} by using only a single leaf of the foliation. 
Together with the scaling condition
\[
q'\le \frac{n-1}{n+1}\,p,
\]
these examples identify the full admissible $(q,p)$-range for positive-definite H\"ormander operators.

\begin{example}\label{ex1}
\rm

In odd dimensions, the standard compression examples for positive-definite H\"ormander operators \cite{minicozzi1997negative,GHI,sogge2018instability} yield the sharp diagonal obstruction for $L^q\to L^p$ estimates. In particular, they show that any such estimate must satisfy
\[
p\ge 2\,\frac{3n+1}{3n-3}.
\]
Combined with the scaling condition
\[
q'\le \frac{n-1}{n+1}\,p,
\]
this gives the full admissible $(q,p)$-range in odd dimensions.

In even dimensions, we record a positive-definite model example that is a \emph{one-leaf} variant of the Guth--Hickman--Iliopoulou compression construction. (For the sharpness of Theorem \ref{theo main5} on manifolds, the same proof works if one uses the construction of Minicozzi--Sogge \cite{minicozzi1997negative} instead.) In the example from \cite{GHI}, the curved Kakeya set concentrates in a neighborhood of a $\frac{n+2}{2}$-dimensional set, which is foliated by $\frac n2$-dimensional leaves, each modeled on the Kakeya compression example in dimension $n-1$. 
Here, in order to impose a further restriction on the admissible range of $q$, we keep only a single leaf. 
This still preserves the wave packet multiplicity, but the ambient concentration set now has dimension $\frac n2$ rather than $\frac{n+2}{2}$. 
As a result, one obtains the sharper off-diagonal necessary condition
\begin{equation}\label{eq:even-sharp-line}
    \frac1q\le \frac{3n-2}{4}-\frac{3n}{2p}.
\end{equation}

More precisely, let $n\ge 4$ be even. 
There exists a positive-definite H\"ormander operator $\mathcal H^\lambda$ and a family of wave packets
\[
f=\sum_T f_T,
\]
with the following properties.

First, $\|f\|_\infty\sim 1$, and the wave packets are indexed by a family of tubes $\mathbb T$ of cardinality
\[
\#\mathbb T\sim \lambda^{\frac{n-2}{2}},
\]
corresponding to localizing the frequency support to a $\lambda^{-1/2}$-slab around a fixed $\tfrac n2$-dimensional leaf in the example of \cite{GHI}. 
In particular, $f$ is supported in frequency space on a set of measure $\sim \lambda^{-1/2}$, so that
\[
\|f\|_{L^q(B_1^{n-1})}^p\sim \lambda^{-p/(2q)}.
\]

Second, for each $T\in \mathbb T$,
\[
|\mathcal H^\lambda f_T(x)|\gtrsim \lambda^{-\frac{n-1}{2}},
\qquad x\in T.
\]
All these tubes are contained in the $\lambda^{1/2}$-neighborhood of a single $\frac n2$-dimensional leaf $S$. 
Hence
\[
|N_{\lambda^{1/2}}(S)\cap B_\lambda^n|
\sim
\lambda^{\frac n2}\bigl(\lambda^{1/2}\bigr)^{\frac n2}
=
\lambda^{\frac{3n}{4}}.
\]

Third, a typical point of $N_{\lambda^{1/2}}(S)\cap B_\lambda^n$ lies in
$\sim \lambda^{\frac{n-2}{4}}$
of the tubes $T\in\mathbb T$. 
Therefore, after assigning random signs and applying Khintchine's inequality, we have
\[
\|\mathcal H^\lambda f\|_{L^p(B_\lambda^n)}^p
\gtrsim
\lambda^{\frac{3n}{4}}
\Bigl(
\lambda^{\frac{n-2}{4}}
\lambda^{-(n-1)}
\Bigr)^{p/2}
=
\lambda^{\frac{3n}{4}-\frac{3n-2}{8}p}.
\]
Consequently, any estimate of the form
\[
\|\mathcal H^\lambda f\|_{L^p(B_\lambda^n)}
\lesssim_\varepsilon
\lambda^\varepsilon \|f\|_{L^q(B_1^{n-1})}
\]
in even dimensions must satisfy
\[
\lambda^{\frac{3n}{4}-\frac{3n-2}{8}p}
\lesssim
\lambda^{-p/(2q)},
\]
which is equivalent to \eqref{eq:even-sharp-line}.
\end{example}
\begin{remark}\label{rem:admissible-range}
The condition
\[
q'\le \frac{n-1}{n+1}\,p
\]
is forced by scaling.

In odd dimensions, Theorem \ref{theo main1} together with Example \ref{ex1} shows that the full admissible range is
\[
p\ge 2\,\frac{3n+1}{3n-3},
\qquad
q'\le \frac{n-1}{n+1}\,p.
\]

In even dimensions, set
\[
p_0:=2\,\frac{3n+2}{3n-2},
\qquad
p_1:=2\,\frac{3n+1}{3n-3}.
\]
By \cite{GHI}, the diagonal estimate holds at the point $(1/p_0,1/p_0)$, while our scaling-line endpoint estimate yields the point
\[
\biggl(\frac1p,\frac1q\biggr)
=
\biggl(\frac1{p_1},\,1-\frac{n+1}{(n-1)p_1}\biggr).
\]
Interpolating these two bounds gives precisely the line 
\[
\frac1q= \frac{3n-2}{4}-\frac{3n}{2p},
\qquad
p_0\le p\le p_1.
\]
For $p\ge p_1$, the scaling condition is stronger. Therefore, in even dimensions, our estimates together with \cite{GHI} yield the full admissible range
\[
p\ge 2\,\frac{3n+2}{3n-2},
\qquad
\frac1q\le \min\biggl\{\frac{3n-2}{4}-\frac{3n}{2p},\,
1-\frac{n+1}{(n-1)p}\biggr\}.
\]
Since Example \ref{ex1} shows that \eqref{eq:even-sharp-line} is necessary, this range is sharp.

\end{remark}

}

\begin{remark}
\label{bourgain-guth-example}
A similar one-leaf construction shows that, in the absence of the positivity assumption \HHH, the Bourgain--Guth \cite{bourgain2011bounds} bound together with the Tomas--Stein theorem yields the full \(L^q\to L^p\) range for general H\"ormander operators satisfying \H and \HH. 

More specifically, for $\mathcal H^\lambda$  satisfying \H and \HH, in odd dimensions, it is well-known (see \cite[Section 2]{bourgain1991lp} and \cite[Section 6]{bourgain1995somenew}) that there is a function $f=\sum_{T\in\TT}f_T$ with $\|f\|_\infty\lesssim1$, expressed as a sum of wave packets indexed by $\TT$ with $\#\TT\sim \lambda^\frac{n-1}{2}$, and a $\frac{n+1}{2}$-dimensional surface $S$ such that $|\mathcal{H}^\lambda f(x)|\gtrsim \lambda^{-\frac{n-1}{4}}$ for all $x\in N_1(S)$.
Since $|N_1(S)\cap B_\lambda^n|\sim \lambda^\frac{n+1}{2}$, this implies that
\[
    \|\mathcal H^\lambda f\|^p_{L^p(B_\lambda^n)}\gtrsim \|f\|^p_{L^q(B_1^{n-1})}
\]
can only hold if $p\geq2\frac{n+1}{n-1}$. 
This example shows that the Tomas--Stein theorem is sharp.

For even $n$, one can construct a similar one-leaf example as in Remark \ref{ex1}, based on the $n-1$-dimensional example above, to obtain the following.
There is a sum of wave packets $f=\sum_{T\in\TT}f_T$ and an $\frac{n}{2}$-dimensional surface $S$ such that the following is true:
\begin{enumerate}
    \item $\#\TT\sim\lambda^{\frac{n-2}{2}}$, and $\|f\|_\infty\sim 1$.
    \item $f$ is supported in a $\lambda^{-1/2}$-neighborhood of the hyperplane $\{\xi_{n-1}=0\}\cap B^{n-1}_{1}$.
    \item $S\subset\R^{n-1}$, and $|\mathcal{H}^\lambda f(x)|\gtrsim \lambda^{-\frac{n}{4}}$ for all $x\in \bar N_1(S)\cap \{x=(\bar x, x_{n}): |x_{n}|\lesssim \lambda^{1/2}\}$, where $x\in\R^n$, $\bar N_1(S)=\{x=(\bar x, x_{n}): \dist_{n-1}(\bar x, S)\leq1\}$, with $\dist_{n-1}$ denoting the distance in $\R^{n-1}$.
\end{enumerate}
Let $\tilde S:=\bar N_1(S)\cap \{x=(\bar x, x_{n}): |x_{n}|\lesssim \lambda^{1/2}\}$.
As a result, we have $|\tilde S\cap B_\lambda^n|\sim \lambda^\frac{n+1}{2}$ and $\|f\|_q^p\sim \lambda^{-p/(2q)}$.
Therefore, 
\[
    \|\mathcal H^\lambda f\|_{L^p(B_\lambda^n)}^p\gtrsim \|f\|_{L^q(B_1^{n-1})}^p
\]
can only be true if
\[
    \frac{1}{q}\leq\frac{n}{2}-\frac{n+1}{p}.
\]
This condition is exactly the one obtained by interpolating the $(L^2, L^{\frac{2(n+1)}{n-1}})$ estimate of  Tomas--Stein and the $(L^\frac{2(n+2)}{n}, L^\frac{2(n+2)}{n})$ estimate of Bourgain--Guth \cite[Theorem 4]{bourgain2011bounds}.

\end{remark}

\bibliography{reference}
\bibliographystyle{alpha}
\end{document}